\documentclass[11pt, reqno]{amsart}

\usepackage{amsmath, amssymb, amsthm, amsfonts, blindtext, cite, dsfont, enumerate, epsfig, float, hyperref,
infwarerr, lipsum, lmodern, marvosym, mathrsfs, mathtools, mathrsfs, mathtools, microtype, relsize,
soul, stmaryrd, stackengine, scalerel, tabularx, tikz, txfonts}
\usepackage[normalem]{ulem}

\microtypesetup{protrusion=true, expansion=true}

\newtheorem{theorem}{Theorem}
\newtheorem{definition}{Definition}
\newtheorem{remark}{Remark}

\newtheorem{corollary}{Corollary}
\newtheorem{proposition}{Proposition}
\newtheorem{lemma}{Lemma}

\DeclareMathOperator{\Entr}{H}

\allowdisplaybreaks

\begin{document}

\title{The Lov\'{a}sz Local Lemma: \\ Foundations and Applications}

\author{Igal Sason}

\begin{abstract}
The Lov\'{a}sz Local Lemma (LLL) is a central tool in probabilistic combinatorics, providing
a sufficient condition under which a finite collection of undesirable events with limited
dependencies can be simultaneously avoided with positive probability. This paper offers a
self-contained expository treatment of the lemma and its strengthened versions, emphasizing
mathematical foundations, conceptual clarity, and applications. We begin with a pedagogically
motivated proof of the LLL based entirely on unconditional probability inequalities.
Particular attention is given to the symmetric form of the lemma and several subsequent
strengthenings. We also discuss a variety of classical applications of both the
symmetric and asymmetric forms of the LLL in combinatorics and graph theory, including bounds
for the edge-disjoint paths problem, satisfiability of Boolean formulas in conjunctive normal
form, lower bounds on diagonal and off-diagonal Ramsey numbers, hypergraph coloring results,
structural properties of directed graphs, and acyclic graph colorings. Additional observations
and refinements are provided throughout. We also introduce the algorithmic framework of Moser
and Tardos, highlighting its constructive counterpart to the LLL, together with an introduction
to the entropy-compression principle. The lopsided LLL, a refinement of the LLL, is presented
along with an application to the Latin transversal problem. We further discuss the cluster-expansion
lemma and its relation to the LLL, and present an alternative treatment of the Latin transversal
problem from the cluster-expansion perspective that yields an improved result. The exposition 
concludes with a high-level overview of the iterated LLL, also known as the semi-random method.
\end{abstract}

\keywords{Lov\'{a}sz Local Lemma; Moser--Tardos algorithm; graph and hypergraph coloring;
probabilistic combinatorics.
\newline \vspace*{0.1cm}
\newline
I. Sason is with the Andrew and Erna Viterbi Faculty of Electrical and Computer Engineering and the Department of Mathematics,
both at the Technion --- Israel Institute of Technology, Haifa 3200003, Israel. Email: \tt{eeigal@technion.ac.il.}}

\maketitle
\markboth{The Lov\'{a}sz Local Lemma: Foundations and Applications}{Igal Sason}

\thispagestyle{empty}

\section{Introduction}
\label{section: introduction}

In many applications of the probabilistic method, one considers a finite
collection of undesirable events $\{A_i\}_{i=1}^n$ and seeks to show that,
with positive probability, none of them occurs. If the sum of their probabilities
is strictly less than one, then the union bound guarantees this. If, instead,
the events are independent and each occurs with probability strictly less than
one, then the probability that none of the events occurs equals the product of
the probabilities of their complements, which is positive.

In typical combinatorial settings, however, the events exhibit significant
dependencies and their probabilities are often not sufficiently small for
direct probabilistic estimates to be effective. Nevertheless, in many
applications, each event depends on only a limited number of others, and the
event probabilities are sufficiently small compared to the reciprocal of the
maximum number of events on which any given event depends. This observation
underlies a fundamental result, namely the \emph{Lov\'{a}sz Local Lemma} (LLL),
introduced by Erd\H{o}s and Lov\'{a}sz \cite{LLL75}. The lemma provides a
sufficient condition under which a finite collection of events with limited
dependencies can be simultaneously avoided with positive probability.

The LLL and its variants have become indispensable tools in probabilistic
combinatorics and theoretical computer science. Over the past several decades,
they have found applications across a broad spectrum of areas, including
probabilistic and extremal combinatorics (see, e.g.,
\cite{LLL75, AlonSpencer2016, AlonMM1996, Bollobas2001, Jukna11,
MolloyR2002, Spencer75, Spencer77, AlonL89, Alon91b, AlonMR91,
KirtisogluO2024, LintW2001, LiLin22, MitzenmacherU17, Knuth22,
FurediK86, Bremaud17, HaeuplerSS11, KolipakaSX12, DengSW04, ChangHLPU2020,
BroderFU94, Vondrak-LN}), satisfiability and computational
complexity \cite{AmbainisKS2012, KratochvilST93, BroderFU94,
HaeuplerSS11, GebauerMSW09, Gebauer2012, GebauerST2016, Moitra19,
Livieratos2020, ChenLWYY2025}, information theory, coding theory, 
communication networks, and group testing
\cite{Schulman93, AlonBEGH16, Gelles17, KeevashKu2006, ConST23, ConST25,
LeightonMR94, LeightonMA99, Huang2025, ChengHLSW19, ArandaF17,
FernandezKMT23, FernandezLM23, FernandezLM25, FernandezMGSLG26,
Yeh2002, DalaiFRV2025, RescignoV2023, RescignoV2024, GarganoRV2020}, 
and statistical physics \cite{ScottS2005, ScottS2006, BissacotFPS2011}. 
For surveys and expository accounts, see \cite{Szegedy2013, Farago21,
GrotschelN2024, Molloy98, Beck2013, Kahn2013, Trotter2013}.

In its classical form, the LLL is \emph{non-constructive}:
it guarantees the existence of an assignment that avoids all the undesirable events, yet
does not, in general, provide an explicit method for finding such an assignment. To describe
the dependence structure among events underlying the LLL, we introduce the following notions.
Throughout, $[n] := \{1,\ldots,n\}$ for $n \in \mathbb{N}$, with $[0] := \varnothing$.

\begin{definition}[Mutual independence]
\label{definition: Mutual independence}
{\em Let $(\Omega,\mathcal{F},\mathbb{P})$ be a probability space, $n \in \mathbb{N}$,
and $A_1, \ldots, A_n \in \mathcal{F}$.
For $i \in [n]$, the event $A_i$ is independent of the $\sigma$-algebra
$\sigma\bigl(\{A_j: j \in [n] \setminus \{i\}\}\bigr)$ if
\begin{align}
\mathbb{P}(A_i \cap B) = \mathbb{P}(A_i)\,\mathbb{P}(B)
\qquad \text{for all } B \in \sigma(A_1,\ldots,A_{i-1},A_{i+1},\ldots,A_n).
\end{align}
If this holds for all $i \in [n]$, then the events $A_1, \ldots, A_n$ are said
to be \emph{mutually independent}.}
\end{definition}

\begin{remark}
\label{remark: sigma-algebra}
{\em For $i \in [n]$, $\sigma\bigl(\{A_j : j \in [n]\setminus\{i\}\}\bigr)$ is called the
\emph{$\sigma$-algebra generated} by the events $A_1, \ldots, A_{i-1}, A_{i+1}, \ldots, A_n$,
and it is defined as the smallest $\sigma$-algebra containing these events.
Since the family $A_1,\ldots,A_{i-1},A_{i+1},\ldots,A_n$ is finite, the $\sigma$-algebra it generates
consists precisely of all events that are obtained from these events by finitely many applications of
unions, intersections, and complements.}
\end{remark}

\begin{definition}[Dependency digraph]
\label{definition: Dependency digraph}
{\em Let $\{A_i\}_{i=1}^n$ be events in a probability space. A directed
graph (digraph) $D = ([n], E)$ is a \emph{dependency digraph}
for $\{A_i\}_{i=1}^n$ if, for all $i \in [n]$, the event $A_i$ is independent
of the $\sigma$-algebra generated by $\{A_j: j \neq i, \, (i,j) \notin E\}$,
where $(i,j)$ denotes an arc from $i$ to $j$.}
\end{definition}

\begin{remark}
\label{remark: mutual independence}
{\em Definition~\ref{definition: Dependency digraph} requires more than pairwise
independence: for each $i \in [n]$, the event $A_i$ must be independent of
the $\sigma$-algebra generated by the events $\{A_j: j \neq i, \, (i,j)\notin E\}$.}
\end{remark}

\begin{remark}[Non-uniqueness of the dependency digraph]
\label{remark: non-unique dependency digraph}
{\em In general, the dependency digraph for the events $\{A_i\}_{i=1}^n$ is not unique.
Indeed, if $D=([n],E)$ is a dependency digraph, then any digraph obtained by adding
arcs to $E$ is also valid, since enlarging $E$ weakens the independence requirement.}
\end{remark}

\begin{theorem}[Lov\'{a}sz Local Lemma (LLL)]
\label{theorem: LLL}
{\em Let $A_1, \ldots, A_n$ be events in a probability space $(\Omega,\mathcal{F},\mathbb{P})$,
and let $D = ([n], E)$ be a dependency digraph for these events. If there exist
$x_1, \ldots, x_n \in [0,1)$ such that
\begin{align}
\label{eq1: condition}
\mathbb{P}(A_i) \leq x_i \prod_{j: (i,j) \in E} (1-x_j), \qquad \forall \, i \in [n],
\end{align}
then
\begin{align}
\label{eq1: LLL}
\mathbb{P}\left( \bigcap_{i=1}^n \overline{A_i} \right) \geq \prod_{i=1}^n (1-x_i).
\end{align}
In particular, the probability that none of the events $\{A_i\}_{i=1}^n$ occurs is positive.}
\end{theorem}
To illustrate the power of the LLL beyond the settings covered by the union bound and independence,
suppose there exists an absolute constant $c < \tfrac12$ such that, for all $i \in [n]$,
\[
\mathbb{P}(A_i) \le c, \qquad \sum_{j: (i,j) \in E} \mathbb{P}(A_j) \le \tfrac14.
\]
Setting $x_i := 2 \, \mathbb{P}(A_i)$ yields $x_i \in [0,1)$, and
\[
\prod_{j: (i,j) \in E} (1-x_j)
\geq 1 - \sum_{j: (i,j) \in E} x_j
= 1 - 2 \sum_{j: (i,j) \in E} \mathbb{P}(A_j)
\geq \tfrac12.
\]
Consequently,
\[
\mathbb{P}(A_i) = \frac{x_i}{2} \le x_i \prod_{j: (i,j) \in E} (1-x_j), \quad \forall\, i \in [n].
\]
By the LLL (Theorem~\ref{theorem: LLL}), since $0 \leq x_i \leq 2c < 1$ for all $i \in [n]$,
it follows that
\begin{align*}
\mathbb{P}\left( \bigcap_{i=1}^n \overline{A_i} \right) & \geq \prod_{i=1}^n (1-x_i) \\
& \geq (1-2c)^n,
\end{align*}
yielding an explicit positive lower bound on the probability that no event in
$\{A_i\}_{i=1}^n$ occurs, with the lower bound decaying exponentially in~$n$.
This extends a statement in \cite[p.~12]{Molloy98}, restricted to the case where
$c = \tfrac18$, where the positivity of this probability is noted.

We briefly review several major algorithmic and structural developments related to the LLL.
Over the years, a variety of important extensions and refinements of the lemma have
been developed, many of which address its historically non-constructive nature.
A central line of research, initiated in 1991 by Beck \cite{Beck91} and
Alon \cite{Alon91}, concerns \emph{algorithmic versions} of the lemma. In this
context, the \emph{variable version of the LLL}, where events are
modeled as functions of independent random variables, plays a fundamental role. A landmark
breakthrough in this direction is due to Moser and Tardos \cite{MoserTardos10}, building
on the earlier breakthrough of Moser \cite{Moser09}, who introduced an efficient
randomized resampling algorithm for avoiding all bad events in this variable
setting. Subsequently, Kolipaka and Szegedy \cite{KolipakaS11} showed that the
Moser--Tardos algorithm remains efficient throughout the full regime characterized by
Shearer’s bound \cite{Shearer85}.

This framework has also led to a refined understanding of dependency structures,
including sharp characterizations of feasibility boundaries. Recent developments have
further revealed gaps between the variable setting and the abstract dependency-graph
formulation of the LLL \cite{HeLS23,HeLLWX26}. Moreover, \cite{HeLLWX26}
establishes improved efficiency guarantees in the variable setting, thereby further
separating it from the abstract LLL framework, for which Shearer’s bound is optimal
under dependency-graph formulations.

The analysis of these algorithms is closely related to the
\emph{entropy-compression method}, which provides a unifying framework for bounding
the probability of long resampling sequences and has evolved into a powerful
and widely used technique in probabilistic combinatorics
\cite{Tao2009,DujmovicJKW16,EsperetP13}. Subsequent works have further extended and
refined this framework; see, e.g., the generalizations of the Moser--Tardos algorithm by
Pegden \cite{Pegden14} and by Harvey and Vondr\'{a}k \cite{HarveyV20}. Deterministic
algorithms for the LLL have also been developed \cite{ChandrasekaranGH13,Harris23}.
Another important direction concerns lopsided variants of the LLL, which replace the 
classical mutual independence assumptions by weaker asymmetric dependency conditions 
\cite{ErdosSpencer91}. Several strengthenings of the LLL, including the lopsided LLL 
and the Cluster-Expansion Lemma, have proved successful in applications where the 
standard LLL does not apply directly (see, e.g., \cite{BissacotFPS2011, MolloyR2002}).

\smallskip 
This paper offers a self-contained expository treatment of the LLL and its strengthenings,
emphasizing mathematical foundations and applications and providing observations and 
refinements throughout. Its main features are as follows:
\begin{enumerate}
\item A reformulated proof of the LLL (Section~\ref{section: LLL and modified proof}), together with a pedagogical motivation
in Remark~\ref{remark: avoiding conditional probabilities}.

\item A treatment of the symmetric version of the LLL (Section~\ref{section: symmetric version}), including strengthened
guarantees on the probability of avoiding all bad events (Corollaries~\ref{corollary: 1/2} and~\ref{corollary: Knuth},
followed by Remark~\ref{remark: comparison of 2 corollaries}).

\item A presentation of classical combinatorial applications of the symmetric and asymmetric LLL (Section~\ref{section: applications}),
together with several observations and refinements, including:
\begin{enumerate}
\item Tightened results in Theorems~\ref{theorem: Edge-disjoint paths} and
\ref{theorem: k-SAT formula}, concerning the problems of edge-disjoint paths
and satisfiability of Boolean formulas, respectively.

\item A slight tightening of Spencer's lower bound on diagonal Ramsey numbers
(Theorem~\ref{theorem: LB R(k,k) - ver4}), together with a refined asymptotic
formulation (Proposition~\ref{prop: ramsey-asymptotic lb} and Remark~\ref{remark: k_0}).
A tightened lower bound on off-diagonal Ramsey numbers, building on Spencer's
application of the (asymmetric) LLL, is also presented (Theorem~\ref{theorem: IS - off-diagonal Ramsey}).

\item Strengthened results on hypergraph colorings
(Theorems~\ref{theorem1: Erdos-Lovasz 75},
\ref{theorem2: Erdos-Lovasz 75},
\ref{theorem3: sharpened Erdos-Lovasz 75},
and Remark~\ref{remark: sharpened Erdos-Lovasz 75}).

\item Strengthened bounds on cycle lengths in digraphs
(Corollaries~\ref{cor: even-length directed cycles},
\ref{cor: cycles divisible by k in regular digraphs}, and  
Remark~\ref{remark: strengthened result - 12.06.26}).

\item An application of the classical LLL from \cite{AlonMR91} leads to a derivation
of a tightened upper bound on the acyclic chromatic number of a graph
(Proposition~\ref{proposition: AlonMR91 tightening}).
\end{enumerate}

\item A presentation of the Moser--Tardos algorithm~\cite{MoserTardos10} for the
constructive proof of the LLL in the variable setting (Section~\ref{section: Moser--Tardos Algorithm}),
and an introduction to the entropy-compression principle (Section~\ref{section: entropy-compression principle}).

\item A presentation of the lopsided LLL, together with its application to the existence of Latin transversals in matrices (Section~\ref{section: Lopsided LLL}).

\item A presentation of the Cluster-Expansion Lemma, its implications, and an alternative treatment of the Latin transversal
problem by that lemma, which yields an improved result (Section~\ref{section: cluster-expansion lemma}).

\item A presentation of the iterated LLL (Section~\ref{section: iterated LLL}), and a brief outlook on open directions (Section~\ref{section: outlook}).
\end{enumerate}

\section{A reformulated proof of the Lov\'{a}sz Local Lemma}
\label{section: LLL and modified proof}

This section presents a reformulated proof of the LLL in Theorem~\ref{theorem: LLL}, followed by its pedagogical
motivation in Remark~\ref{remark: avoiding conditional probabilities}.

\begin{proof}
We present a proof based entirely on unconditional probability inequalities.
In particular, no step requires assuming that a conditioning event has positive probability.

\medskip
\noindent\textbf{Step 1: An auxiliary lemma.}
For $S \subseteq [n]$, define
\begin{align}
\label{eq: F(S)}
F(S) := \bigcap_{j\in S} \, \overline{A_j}.
\end{align}

\begin{lemma}
\label{lemma: LLL-unconditional}
{\em Assume that the conditions of Theorem~\ref{theorem: LLL} hold. Then, for every
$S \subset [n]$ and $i \in [n] \setminus S$,
\begin{align}
\label{eq: LLL-unconditional}
\mathbb{P}\bigl(A_i \cap F(S)\bigr) \leq x_i\,\mathbb{P}\bigl(F(S)\bigr).
\end{align}}
\end{lemma}

\smallskip
\noindent\emph{Proof of Lemma~\ref{lemma: LLL-unconditional}.}
We prove \eqref{eq: LLL-unconditional} by induction on $|S|$.

\smallskip
\noindent\emph{Base case.}
If $S=\varnothing$, then $F(S)=\Omega$, and \eqref{eq: LLL-unconditional} reduces to
$\mathbb{P}(A_i)\leq x_i$. Indeed, by \eqref{eq1: condition}
\[
\mathbb{P}(A_i) \leq x_i \prod_{j:(i,j) \in E} (1-x_j) \leq x_i,
\]
where the last inequality holds since each factor satisfies $1-x_j \in [0,1]$.

\smallskip
\noindent\emph{Induction hypothesis.}
Fix an integer $m \geq 1$ and assume that \eqref{eq: LLL-unconditional} holds for every
$S' \subset [n]$ with $|S'| < m$ and every $i' \in [n] \setminus S'$.

\smallskip
\noindent\emph{Induction step.}
Let $S\subset[n]$ be an arbitrary set with $|S|=m$, and let $i \in [n] \setminus S$.
Define
\begin{align}
\label{eq: S1S2}
S_1 := \{j\in S : (i,j) \in E\}, \qquad S_2 := S \setminus S_1 .
\end{align}
If $S_1=\varnothing$, then $(i,j)\notin E$ and $j \neq i$ for all $j \in S$. By
Definition~\ref{definition: Dependency digraph}, the event $A_i$ is independent
of the $\sigma$-algebra generated by $\{A_j : j \in S\}$, and in particular it is
independent of $F(S)=\bigcap_{j \in S}\overline{A_j}$.
Therefore,
\[
\mathbb{P}\bigl(A_i \cap F(S)\bigr)
=\mathbb{P}(A_i) \, \mathbb{P}\bigl(F(S)\bigr)
\leq x_i \, \mathbb{P}\bigl(F(S)\bigr),
\]
and \eqref{eq: LLL-unconditional} follows.

Assume now that $S_1 \neq \varnothing$.
By \eqref{eq: S1S2}, $S_2 \subseteq \{j: (i,j)\notin E\}$. Since $i \notin S$ and $S_2 \subseteq S$,
we have $i \notin S_2$.
Therefore, $S_2 \subseteq \{j: j \neq i, \, (i,j)\notin E\}$. By Definition~\ref{definition: Dependency digraph},
$A_i$ is independent of the $\sigma$-algebra generated by $\{A_j : j \in S_2\}$; in particular,
$A_i$ is independent of $F(S_2)$.
As $S_2 \subseteq S$, we have $F(S)\subseteq F(S_2)$ so
\begin{align}
\mathbb{P}\bigl(A_i \cap F(S)\bigr)
&\leq \mathbb{P}\bigl(A_i \cap F(S_2)\bigr) \nonumber \\
&= \mathbb{P}(A_i)\;\mathbb{P}\bigl(F(S_2)\bigr) \nonumber \\
\label{eq: num-ub}
&\leq x_i \prod_{j:(i,j)\in E}(1-x_j) \; \mathbb{P}\bigl(F(S_2)\bigr),
\end{align}
where the last inequality holds by \eqref{eq1: condition}.
We next derive a lower bound on $\mathbb{P}(F(S))$ in terms of $\mathbb{P}(F(S_2))$.
Let $S_1=\{j_1,\dots,j_r\}$ and, for $t\in[r]$, let
\[
T_t := S_2\cup\{j_1,\dots,j_{t-1}\}.
\]
Note that $|T_t|=|S_2|+t-1\leq |S|-1$, hence $|T_t|<|S|$.
Define, for $t=0,1,\dots,r$, the events
\[
G_t := F(S_2)\cap \bigcap_{s=1}^t \overline{A_{j_s}},
\]
so that $G_0=F(S_2)$ and $G_r=F(S)$ (recall that $S = S_1 \, \dot{\cup} \, S_2$ is a
disjoint union of $S_1$ and $S_2$). Then, $G_t = \overline{A_{j_t}} \cap G_{t-1}$, and therefore
\begin{align}
\label{eq: G-recursion}
\mathbb{P}(G_t)=\mathbb{P}(G_{t-1})-\mathbb{P}(A_{j_t}\cap G_{t-1}),
\qquad t\in[r].
\end{align}
Since $G_{t-1}=F(T_t)$ and $j_t \notin T_t$, the induction hypothesis applied to the pair $(i',S')=(j_t,T_t)$ gives
\[
\mathbb{P}(A_{j_t}\cap G_{t-1})
=\mathbb{P}\bigl(A_{j_t}\cap F(T_t)\bigr)
\leq x_{j_t}\,\mathbb{P}\bigl(F(T_t)\bigr)
=x_{j_t}\,\mathbb{P}(G_{t-1}).
\]
Substituting into \eqref{eq: G-recursion} yields
\begin{align}
\mathbb{P}(G_t) \geq (1-x_{j_t})\,\mathbb{P}(G_{t-1}),
\qquad t\in[r].
\end{align}
Iterating for $t=1,2,\dots,r$, we obtain
\begin{align}
\label{eq: denom-lb}
\mathbb{P}\bigl(F(S)\bigr) = \mathbb{P}(G_r)
\geq \prod_{t=1}^r (1-x_{j_t})\;\mathbb{P}(G_0)
= \prod_{j \in S_1}(1-x_j) \; \mathbb{P}\bigl(F(S_2)\bigr).
\end{align}
Since $S_1 \subseteq \{j:(i,j)\in E\}$ and each factor $1-x_j$ lies in $[0,1]$, it follows that
\begin{align}
\label{eq: prod-compare}
\prod_{j\in S_1}(1-x_j) \geq \prod_{j:(i,j)\in E}(1-x_j).
\end{align}
Combining \eqref{eq: num-ub}, \eqref{eq: denom-lb}, and \eqref{eq: prod-compare} gives
\begin{align*}
\mathbb{P}\bigl(A_i\cap F(S)\bigr)
& \leq x_i \prod_{j:(i,j)\in E} (1-x_j)\;\mathbb{P}\bigl(F(S_2)\bigr) \\
& \leq x_i \prod_{j \in S_1} (1-x_j)\;\mathbb{P}\bigl(F(S_2)\bigr) \\
& \leq x_i\;\mathbb{P}\bigl(F(S)\bigr),
\end{align*}
which is \eqref{eq: LLL-unconditional}. This completes the induction and proves
Lemma~\ref{lemma: LLL-unconditional}.
\hfill $\blacksquare$

\medskip
\noindent
\textbf{Step 2: Concluding the proof of the LLL.}
For $k=0,1,\dots,n$, define
\begin{align}
\label{eq: F_k}
F_k := \bigcap_{j=1}^k \overline{A_j},
\quad \text{with  } F_0=\Omega.
\end{align}
Let $k \in [n]$ and $S=\{1,\dots,k-1\}$ (if $k=1$, then $S = \varnothing$).
By \eqref{eq: F(S)} and \eqref{eq: F_k}, we have $F(S)=F_{k-1}$.
Applying Lemma~\ref{lemma: LLL-unconditional} with this choice of $S$ and $i=k$
(noting that $i \in [n] \setminus S$), we obtain
\[
\mathbb{P}(A_k\cap F_{k-1}) \leq x_k\,\mathbb{P}(F_{k-1}),
\]
and therefore
\begin{align}
\label{eq: P(F_k)}
\mathbb{P}(F_k)
&=\mathbb{P}(F_{k-1})-\mathbb{P}(A_k\cap F_{k-1}) \nonumber \\
&\geq (1-x_k)\,\mathbb{P}(F_{k-1}),
\end{align}
with $\mathbb{P}(F_0) = 1$.
Iterating \eqref{eq: P(F_k)} for $k=1,2,\dots,n$ and using \eqref{eq: F_k}, we obtain
\[
\mathbb{P}\,\Biggl(\bigcap_{i=1}^n \overline{A_i}\Biggr)
=\mathbb{P}(F_n)
\geq \prod_{i=1}^n (1-x_i),
\]
which is \eqref{eq1: LLL}. In particular, since $x_i \in [0,1)$ for all $i \in [n]$, it follows that
\[
\mathbb{P}\,\left(\bigcap_{i=1}^n \overline{A_i}\right)>0.
\]
\end{proof}

\begin{remark}[On avoiding conditioning assumptions]
\label{remark: avoiding conditional probabilities}
{\em A common presentation of the LLL proves a variation of
Lemma~\ref{lemma: LLL-unconditional}, namely,
\begin{align}
\label{eq: cond. prob. inequality}
\mathbb{P}\left(A_i \,\big|\, \bigcap_{j\in S}\overline{A_j}\right) \leq x_i,
\qquad \forall\, S \subset [n], \; \; i \in [n] \setminus S,
\end{align}
by manipulating conditional probabilities via identities such as
\begin{align}
\label{eq: cond. prob. equality}
\mathbb{P}\,(A \mid B\cap C)=\frac{\mathbb{P}\,(A\cap B \mid C)}{\mathbb{P}\,(B \mid C)},
\end{align}
and expressing them in terms of conditional probabilities of the form
\[
\mathbb{P}\left(A_i \cap \bigcap_{j\in S \setminus S'} \overline{A_j} \; \big| \;
\bigcap_{j\in S'} \overline{A_j}\right)
\quad \text{and} \quad
\mathbb{P}\left(A_i \; \big| \; \bigcap_{j\in S'} \overline{A_j}\right),
\qquad S' \subset S.
\]
See for example \cite[pp.~616--617]{LLL75}, \cite[pp.~70--72]{AlonSpencer2016},
\cite[pp.~21--23]{Bollobas2001}, \cite[pp.~100--103]{Bremaud17},
\cite[pp.~280--282]{Jukna11}, \cite[pp.~111--114]{LiLin22}, \cite[pp.~30--31]{LintW2001},
\cite[pp.~147--150]{MitzenmacherU17}, \cite[pp.~226--228]{MolloyR2002}, \cite{Spencer75},
\cite{Spencer77}, and \cite[p.~266]{Knuth22}.
In standard proofs, such conditional probabilities are manipulated within
an inductive argument. However, their definition requires that the corresponding
conditioning events have positive probability, a fact whose validity is established
only within the proof itself. Hence, intermediate steps involve expressions
whose validity depends on properties that are justified only a posteriori.
The approach adopted here avoids introducing conditional probabilities altogether.
Instead, we work with inequalities such as
\begin{align}
\mathbb{P}\left(A_i \cap \bigcap_{j\in S}\overline{A_j}\right)
\leq x_i\,\mathbb{P}\left(\bigcap_{j\in S}\overline{A_j}\right),
\end{align}
which remain well-defined regardless of whether
$\mathbb{P}(\cap_{j\in S} \, \overline{A_j})$ is positive or zero.
This leads to a self-contained argument in which the positivity of
$\mathbb{P}(\overline{A_1} \cap \dots \cap \overline{A_n})$
follows without requiring separate justification of the positivity
of intermediate conditioning events.

An alternative proof of the (asymmetric) LLL that does not rely on
conditional probabilities, brought to our attention by one of the
anonymous referees, is presented in the lecture notes of Vondr\'{a}k
\cite{Vondrak-LN}.}
\end{remark}

\section{Symmetric Lov\'{a}sz Local Lemma}
\label{section: symmetric version}
In a wide range of applications, the events $\{A_i\}_{i=1}^n$ satisfy uniform bounds on their probabilities
and dependencies, allowing the conditions in Theorem~\ref{theorem: LLL} to be simplified. This gives the next result.

\begin{theorem}[Symmetric version of the Lov\'{a}sz Local Lemma (LLL)]
\phantomsection
\label{theorem: LLL - symmetric}
{\em Let $A_1, \ldots, A_n$ be events in an arbitrary probability space. Suppose that, for every $i \in [n]$,
the event $A_i$ is independent of the $\sigma$-algebra generated by all the remaining events, except for at
most $d \geq 1$ of them, and assume that $\mathbb{P}(A_i) \leq p$ for all $i \in [n]$. If $ep(d+1) \leq 1$, then
\begin{align}
\mathbb{P}\left( \bigcap_{i=1}^n \overline{A_i} \right) \ & \geq \ \biggl( \frac{d}{d+1} \biggr)^n \nonumber \\
\label{eq2: LLL}
\ & > \ e^{-\frac{n}{d}}.
\end{align}
In particular, the probability that none of the events occurs is positive.}
\end{theorem}

\begin{remark}
{\em If $d=0$, then the events $\{A_i\}_{i=1}^n$ are mutually independent, so
\[
\mathbb{P}\left( \bigcap_{i=1}^n \overline{A_i} \right) = \prod_{i=1}^n \mathbb{P}(\overline{A_i})
\geq (1-p)^n > 0.
\]
Thus, the probability that none of the events $\{A_i\}_{i=1}^n$ occurs is positive whenever $p<1$,
in contrast to the stronger condition $p \leq \frac{1}{e}$ that would result from formally substituting
$d=0$ into the condition $ep(d+1)\leq 1$.}
\end{remark}

\begin{proof}
Let $d \geq 1$, and set $x_j := \frac1{d+1}$ for all $j \in [n]$.
By the assumption of Theorem~\ref{theorem: LLL - symmetric}, there exists a dependency digraph $D = ([n], E)$
in which every vertex has outdegree at most $d$ (see Definition~\ref{definition: Dependency digraph}).
Fix such a dependency digraph $D$ for $\{A_i\}_{i=1}^n$. Then, for all $i \in [n]$,
\begin{align}
\label{eq1: 05.03.26}
\mathbb{P}(A_i) &\leq \frac1{(d+1)e} \\
\label{eq2: 05.03.26}
&\leq \frac1{d+1} \; \biggl(1+\frac1d\biggr)^{-d} \\
\label{eq3: 05.03.26}
&= \frac1{d+1} \; \biggl(1-\frac1{d+1}\biggr)^d  \\
\label{eq4: 05.03.26}
&\leq x_i \prod_{j: (i,j) \in E} (1-x_j),
\end{align}
where \eqref{eq1: 05.03.26} holds by the assumption of the theorem;
\eqref{eq2: 05.03.26} follows from the fact that the sequence $\bigl\{ \bigl(1+\frac1k\bigr)^k \bigr\}_{k \in \mathbb{N}}$
is monotonically increasing and converges to $e$ as $k \to \infty$; \eqref{eq3: 05.03.26} follows by straightforward
algebra; finally, \eqref{eq4: 05.03.26} holds by the choice $x_i = \frac1{d+1}$ for all $i \in [n]$ and since
the outdegrees of the considered dependency digraph are at most $d$. It then follows from Theorem~\ref{theorem: LLL} that
\begin{align*}
\mathbb{P}\left( \bigcap_{i=1}^n \overline{A_i} \right) \geq \prod_{i=1}^n (1-x_i)
= \biggl( \frac{d}{d+1} \biggr)^n > e^{-\frac{n}{d}},
\end{align*}
where the last inequality holds since $\left(1+\frac1d\right)^d < e$.
\end{proof}

\noindent
By skipping inequality \eqref{eq1: 05.03.26} and starting instead from inequality \eqref{eq2: 05.03.26}, the next
sharper symmetric criterion is obtained.
\begin{corollary}[Spencer's bound]
\label{corollary: Spencer's bound}
{\em Let $A_1, \ldots, A_n$ be events in an arbitrary probability space. Suppose that, for every $i \in [n]$,
the event $A_i$ is independent of the $\sigma$-algebra generated by all the remaining events, except for at
most $d \geq 1$ of them, and assume that $\mathbb{P}(A_i) \leq p$ for all $i \in [n]$. If
\begin{align}
\label{eq1: Spencer}
p \leq \frac{d^d}{(d+1)^{d+1}},
\end{align}
then
\begin{align}
\label{eq2: Spencer}
\mathbb{P}\left( \bigcap_{i=1}^n \overline{A_i} \right) \ \geq \ \biggl( \frac{d}{d+1} \biggr)^n > e^{-\frac{n}{d}}.
\end{align}}
\end{corollary}

\begin{corollary}
\label{corollary: 1/2}
{\em If $e p \bigl(d + \tfrac12 \bigr) \leq 1$, then \eqref{eq2: Spencer} holds.}
\end{corollary}
\begin{proof}
The condition $ep \bigl(d + \tfrac12 \bigr) \leq 1$ yields \eqref{eq1: Spencer},
as justified in Appendix~\ref{appendix: 1/2}. Hence, \eqref{eq2: Spencer} holds by
Corollary~\ref{corollary: Spencer's bound}.
\end{proof}

\begin{remark}
\label{remark: history and novelty}
{\em The original version of the LLL, introduced by Erd\H{o}s and
Lov\'asz in~\cite{LLL75}, asserts that
$\mathbb{P}\Bigl(\bigcap_{i=1}^n \overline{A_i}\Bigr) > 0$
provided that $4pd \leq 1$, where $p$ is an upper bound on
the probability of each event $A_i$, and $d$ is an upper
bound on the number of dependencies of each event.
The standard sharpened formulation of the symmetric LLL, stated in
Theorem~\ref{theorem: LLL - symmetric}, establishes the condition
$ep(d+1) \leq 1$.
A further refinement, due to Spencer~\cite[Theorem~1.4]{Spencer77},
shows that the same conclusion holds under the weaker condition
\eqref{eq1: Spencer}. Theorem~\ref{theorem: LLL - symmetric} and
Corollary~\ref{corollary: Spencer's bound} also provide the explicit
lower bound $\Bigl(\frac{d}{d+1}\Bigr)^n$ on the probability that none
of the events $\{A_i\}_{i=1}^n$ occurs, thereby strengthening the
classical assertion that this probability is merely positive. By
Corollary~\ref{corollary: 1/2}, the lower bound in~\eqref{eq2: Spencer}
holds, in particular, if $ ep \, \Bigl(d + \tfrac12 \Bigr) \leq 1 $.}
\end{remark}

In \cite{Spencer77}, Spencer defined $f(d)$ as the supremum of the set of all
$x \in [0,1)$ such that, whenever $\mathbb{P}(A_i) \leq x$ for all
$i \in [n]$ and each event $A_i$ is independent of the $\sigma$-algebra
generated by all but at most $d \geq 1$ of the remaining events, then
$\mathbb{P}\left(\bigcap_{i=1}^n \overline{A_i}\right) > 0$.
He further asked whether the limit $\underset{d \to \infty}{\lim} \, d \, f(d)$ exists,
and if so, what its value is. The following theorem of
Shearer~\cite[Theorem~2]{Shearer85} answers this question.
\begin{theorem}[Shearer]
\label{theorem: Shearer}
{\em For every integer $d \geq 1$,
\begin{align}
\label{eq: f}
f(d) =
\begin{cases}
\hspace*{0.5cm} \tfrac12, & \text{if } d=1, \\[0.2cm]
\frac{(d-1)^{d-1}}{d^d}, & \text{if } d \geq 2.
\end{cases}
\end{align}
Consequently,
\begin{align}
\label{eq: limit}
\lim_{d \to \infty} \, d f(d) = \frac{1}{e},
\end{align}
and the constant $e$ in the condition $ep(d+1) \leq 1$ of Theorem~\ref{theorem: LLL - symmetric} is asymptotically best possible.}
\end{theorem}

\begin{corollary}[Problem~319 of \cite{Knuth22}] \label{corollary: Knuth}
{\em Let $A_1, \ldots, A_n$ be events in an arbitrary probability space. Suppose that, for all $i \in [n]$,
the event $A_i$ is independent of the $\sigma$-algebra generated by all the remaining events, except for at
most $d \geq 1$ of them, and $\mathbb{P}(A_i) \leq p$. If $epd \leq 1$, then
$\mathbb{P}\left( \bigcap_{i=1}^n \overline{A_i} \right) > 0$.}
\end{corollary}
\begin{proof}
For every $d \in \mathbb{N}$, we have $\frac{1}{ed} < f(d)$. Indeed, if $d=1$, then
$\frac1e < \frac12 = f(1)$, and if $d \geq 2$, then
\begin{align*}
\frac{1}{ed} &< \frac1d \, \biggl(1 + \frac1{d-1}\biggr)^{-(d-1)}
= \frac{(d-1)^{d-1}}{d^d} = f(d).
\end{align*}
Hence, if $epd \leq 1$, then $p < f(d)$, and the result follows
from Theorem~\ref{theorem: Shearer}.
\end{proof}

\begin{remark}
\label{remark: comparison of 2 corollaries}
{\em Corollary~\ref{corollary: 1/2} is not implied by Corollary~\ref{corollary: Knuth}. Indeed,
under the condition $epd \leq 1$ of Corollary~\ref{corollary: Knuth}, the probability that 
none of the events $\{A_i\}_{i=1}^n$ occurs is only guaranteed to be positive for a dependency 
digraph with maximum degree at most $d$. By contrast, under the stronger condition 
$ep \, \Bigl( d+\tfrac12 \Bigr) \leq 1$ of Corollary~\ref{corollary: 1/2}, this
probability is lower-bounded by $e^{-n/d}$ whenever $d \geq 1$.  
For an undirected dependency graph on $n$ vertices, where the maximum degree and its upper 
bound $d$ may exceed those of a corresponding dependency digraph, Vaccaro recently communicated 
to us a proof that if $epd \leq 1$, then the probability that none of the events occurs is at 
least $e^{-n/d}$ whenever $d \geq 1$ \cite{Vaccaro2026}.}
\end{remark}

\section{Combinatorial applications of the Lov\'{a}sz Local Lemma}
\label{section: applications}

This section presents several classical applications of the symmetric and asymmetric LLL in probabilistic combinatorics
and graph theory. These applications include bounds for the problem of edge-disjoint paths
(Section~\ref{subsection: Edge-disjoint paths}), the satisfiability problem for Boolean formulas
in conjunctive normal form (Section~\ref{subsection: k-SAT problems}), lower bounds on diagonal
and off-diagonal Ramsey numbers (Sections~\ref{subsection: diagonal Ramsey numbers}
and~\ref{subsection: off-diagonal Ramsey numbers}, respectively), hypergraph coloring results
(Section~\ref{subsection: Coloring hypergraphs}), the existence of directed cycles with
prescribed modular length (Section~\ref{subsection: Length of cycles in directed graphs}), and
an upper bound on the acyclic chromatic number of graphs (Section~\ref{subsection: Acyclic coloring of graphs}).
These applications highlight the versatility of the LLL as a tool for probabilistic existence proofs,
while strengthened results and additional observations are incorporated throughout this section.

\subsection{Edge-disjoint paths}
\label{subsection: Edge-disjoint paths}
Assume that in a communication network, $n$ pairs of users need to communicate via edge-disjoint paths,
and that for each pair there are at least a given number of candidate paths. Using the symmetric version
of the LLL, one can show that if, for every two distinct pairs of users, each candidate path of one pair
intersects (in an edge) only a limited number of candidate paths of the other pair, then there exists a
choice of edge-disjoint paths connecting all $n$ pairs. The following result makes this statement precise
and provides a strengthened version of \cite[Theorem~6.12]{MitzenmacherU17}.

\begin{theorem}[Edge-disjoint paths]
\label{theorem: Edge-disjoint paths}
{\em Let $G=(V,E)$ be a graph, and let $\{\{x_i,y_i\}\}_{i=1}^n$ be different unordered pairs of distinct
vertices. For each $i \in [n]$, let $\mathcal{Q}_i$ be a set of paths connecting $x_i$ and $y_i$, where
$|\mathcal{Q}_i| \geq m$ for some fixed $m \in \mathbb{N}$.
Assume that for all distinct $i, j \in [n]$, each path in $\mathcal{Q}_i$ shares
an edge with at most $k$ paths in $\mathcal{Q}_j$. If
\begin{align}
\label{eq: 2.5.26}
\begin{cases}
k<m,  & n=2, \\[2mm]
e \cdot \frac{k}{m} \cdot (2n-4) \leq 1,  & n \geq 3,
\end{cases}
\end{align}
then there exist paths $P_i \in \mathcal{Q}_i$, for all $i \in [n]$, such that
$P_1, \dots, P_n$ are pairwise edge-disjoint.}
\end{theorem}

\begin{proof}
For each $i \in [n]$, let the path $P_i$ be chosen independently and uniformly at random
from $\mathcal{Q}_i$. Consider the bad events $A_{i,j} = \{ P_i \text{ and } P_j \text{ share an edge}\}$,
for all pairs $(i,j)$ such that $1 \leq i < j \leq n$.

\smallskip

\noindent
{\em Step 1: Bounding the probability of a bad event.}
Fix $i,j \in [n]$ with $i<j$. Once the path $P_i$ is chosen, it shares an edge with at most $k$
paths in $\mathcal{Q}_j$. Since $P_j$ is chosen independently and uniformly at random from $\mathcal{Q}_j$,
where $|\mathcal{Q}_j| \geq m$, it follows that
\begin{align}
\mathbb{P}(A_{i,j}) \leq \frac{k}{m}.
\end{align}
For $n=2$, there is only one bad event, namely that the two selected paths share an edge. Since $\mathbb{P}(A_{1,2}) \leq \frac{k}{m}$,
the sharp sufficient condition for the existence of a choice of two edge-disjoint paths in this case is $k<m$. Conversely, if $k=m$, 
then the conclusion may fail; for example, this occurs if every path in $\mathcal{Q}_1$ intersects every path in $\mathcal{Q}_2$. 

\smallskip 
\noindent 
We next consider the case where $n \geq 3$.

\noindent
{\em Step 2: The dependency structure.}
Each event $A_{i,j}$ depends only on the random choices $P_i$ and $P_j$.
Since the random paths $P_1,\dots,P_n$ are mutually independent, it follows that
$A_{i,j}$ is independent of the $\sigma$-algebra generated by all the remaining
bad events, except possibly for those events $A_{r,s}$ for which
$\{r,s\} \cap \{i,j\} \neq \varnothing$.
These excluded events are precisely those of the form $A_{i,\ell}$ or $A_{\ell,i}$, and
$A_{j,\ell}$ or $A_{\ell,j}$, where $\ell \in [n] \setminus \{i,j\}$. Other than $A_{i,j}$,
there are exactly $n-2$ bad events involving the index $i$, and similarly $n-2$ bad events
involving the index $j$. Hence, every bad event $A_{i,j}$ is independent of the $\sigma$-algebra
generated by all but at most $2n-4$ of the remaining bad events.

\smallskip

\noindent
{\em Step 3: Applying the symmetric LLL.}
For $n \geq 3$, we apply Corollary~\ref{corollary: Knuth} with
\begin{align}
p = \frac{k}{m}, \qquad d = 2n-4.
\end{align}
If $e p d \leq 1$, then with positive probability no bad event $A_{i,j}$ (for $1 \leq i < j \leq n$)
occurs. This condition coincides with \eqref{eq: 2.5.26}, which holds by assumption. Therefore,
there exist pairwise edge-disjoint paths $P_i \in \mathcal{Q}_i$ for all $i \in [n]$.
\end{proof}

\subsection{Satisfibability problems}
\label{subsection: k-SAT problems}
The $k$-SAT problem is a fundamental object in probabilistic combinatorics, concerned with the satisfiability
of Boolean formulas in conjunctive normal form (CNF), where a formula is expressed as an AND of clauses, each
clause being an OR of literals (variables or their negations), and each clause contains exactly $k$ literals.
In this setting, each clause can be associated with a bad event, namely that the clause is not satisfied. By
endowing the space of assignments with a suitable probability measure, one obtains a collection of (typically
dependent) bad events, making $k$-SAT a natural framework for applications of the LLL. The lemma then provides
sufficient conditions under limited dependency to guarantee the existence of an assignment that avoids all such
events, and hence satisfies the formula. The following result strengthens a result appearing in
\cite[Lemma~19.8]{Jukna11}, \cite[Theorem~6.13]{MitzenmacherU17}, and \cite[Theorem~3.1]{Molloy98}.

\begin{theorem}[Satisfiability criterion for a $k$-SAT formula]
\label{theorem: k-SAT formula}
{\em A $k$-SAT formula is satisfiable if no variable appears in more than $\frac{2^k}{ek}$ clauses.
Furthermore, for such a $k$-SAT formula with $n$ clauses, a uniformly random assignment satisfies
the formula with probability at least $\exp\bigl(-en 2^{-k}\bigr)$.}
\end{theorem}

\begin{proof}
Consider a $k$-SAT formula in which each variable appears in at most $\frac{2^k}{ek}$ clauses, and
choose a random assignment by setting each variable independently to $0$ or $1$, each with probability
$\tfrac{1}{2}$. For each clause $C$, let $A_C$ be the bad event that $C$ is not satisfied.
Since $C$ contains $k$ literals, all of which must evaluate to $0$ for $A_C$ to occur, we have
\begin{align}
p:= \mathbb{P}(A_C) = 2^{-k}.
\end{align}

\smallskip
\noindent
Each event $A_C$ is independent of the $\sigma$-algebra generated by all the remaining
bad events, except for those events whose clause shares a variable with $C$. Each of the
$k$ variables appearing in $C$ appears in at most $\frac{2^k}{ek}$ clauses. Hence, excluding
$C$, the number of clauses sharing at least one variable with $C$ is at most
\begin{align}
\label{eq1: 10.06.26}
d \leq k \cdot \frac{2^k}{ek} - 1 = \frac{2^k}{e} - 1.
\end{align}
Therefore, $A_C$ is independent of the $\sigma$-algebra generated by all but at most $d$ of
the remaining bad events. By Theorem~\ref{theorem: LLL - symmetric}, since $ep(d+1) \leq 1$,
it follows that with positive probability no bad event occurs.
This implies that there exists an assignment that satisfies all clauses.

Finally, we apply the asymmetric LLL (Theorem~\ref{theorem: LLL}) to derive the claimed lower bound on 
the probability that a uniformly random assignment satisfies a $k$-SAT formula with $n$ clauses.
Let $x := 1-e^{-a}$ with $a:=ep = e 2^{-k}$. Consider the dependency graph whose vertex set 
consists of all bad events, and where any two vertices are adjacent if and only if their corresponding 
clauses share at least one variable. Then, the degree of every vertex is at most 
$d \leq \frac{2^k}{e} - 1 = \frac1{a}-1$. Consequently,  
\begin{align}
\label{eq1: 23.06.2016}
p = 2^{-k} = \frac{a}{e} \leq \frac{e^a-1}{e} = (1-e^{-a}) \, e^{-1+a} = x(1-x)^{\frac{1}{a}-1} \leq x (1-x)^d.
\end{align}
Thus the hypotheses of the LLL are satisfied, and therefore 
\begin{align}
\label{eq2: 23.06.2016}
\mathbb{P}\left( \bigcap_{C} \overline{A_C} \right) \geq (1-x)^n = \exp(-en 2^{-k}).
\end{align}
\end{proof}

\subsection{Lower bounds on the diagonal Ramsey numbers}
\label{subsection: diagonal Ramsey numbers}

A fundamental principle of combinatorics is that complete disorder cannot persist in sufficiently large systems.
Ramsey theory provides a precise formulation of this phenomenon, and Ramsey numbers measure the threshold beyond
which prescribed structures must inevitably appear (see, e.g., \cite{LiLin22}).

\begin{definition}[Ramsey numbers]
\label{definition: Ramsey numbers}
{\em Let $k,\ell \geq 2$ be integers. The \emph{Ramsey number} $R(k,\ell)$ is the smallest integer $n \geq 2$
such that every $2$-coloring of the edges of the complete graph $K_n$ contains either a monochromatic copy of $K_k$
in the first color or a monochromatic copy of $K_\ell$ in the second color.
\begin{itemize}
\item The \emph{diagonal Ramsey numbers} are the numbers $R(k,k)$; equivalently, $R(k,k)$ is the smallest integer $n \geq 2$
such that every graph on $n$ vertices contains either a clique or an independent set on $k$ vertices.
\item The \emph{off-diagonal Ramsey numbers} are all the remaining numbers $R(k,\ell)$ with $k \neq \ell$.
\end{itemize}}
\end{definition}

This subsection revisits the lower bound on the diagonal Ramsey numbers derived in \cite{Spencer75},
relying on the symmetric version of the LLL (Theorem~\ref{theorem: LLL - symmetric}). Theorem~\ref{theorem: LB R(k,k) - ver3} restates
\cite[Theorem~2]{Spencer75}. A slight sharpening, based on the same argument, is presented in Theorem~\ref{theorem: LB R(k,k) - ver4},
and the identical asymptotic behavior of the two lower bounds is analyzed in Proposition~\ref{prop: ramsey-asymptotic lb},
which refines \cite[Corollary~1]{Spencer75}.

\begin{theorem} \label{theorem: LB R(k,k) - ver3}
{\em Let $k \geq 2$ be an integer. If
\begin{align}
\label{eq0: 30.05.26}
e \, \binom{k}{2} \, \binom{n-2}{k-2} \, 2^{1-\binom{k}{2}} \leq 1,
\end{align}
for some integer $n \geq 2$, then $R(k,k) > n$.}
\end{theorem}
\begin{proof}
Color the edges of the complete graph $K_n$ using two colors, uniformly at random
and independently. For a set $S$ of $k$ vertices, let $A_{S}$ be the event that
all edges with both endpoints in $S$ are monochromatic. Then,
$\mathbb{P}(A_S) = 2^{1-\binom{k}{2}}$.
Moreover, $A_S$ is independent of the $\sigma$-algebra generated by the events
\[
\bigl\{ A_T: \, T \subseteq [n], \ |T|=k, \ |T \cap S| \leq 1 \bigr\},
\]
since, whenever $|T \cap S| \leq 1$, the events $A_S$ and $A_T$ depend on disjoint
sets of independent edge-color variables.
Thus, the only events $A_T$, with $T \neq S$, that may depend on $A_S$ are those
for which $|T \cap S| \geq 2$. Their number is at most $\binom{k}{2} \, \binom{n-2}{k-2} - 1$.
Indeed, for each pair $\{i,j\}\subseteq S$, let
\[
F_{\{i,j\}} := \bigl\{ T \subseteq [n]: \; |T|=k, \; \{i,j\} \subseteq T \bigr\}.
\]
Then,
\[
|F_{\{i,j\}}| = \binom{n-2}{k-2}.
\]
Clearly,
\[
\bigl\{ T \subseteq [n], \; |T|=k, \; |T \cap S| \geq 2, \, T \neq S \bigr\}
= \left( \bigcup_{\{i,j\} \subseteq S} F_{\{i,j\}} \right) \setminus \{S\},
\]
so, by subadditivity of cardinality,
\begin{align*}
d: = \Bigl| \bigl\{ T \subseteq [n], \; |T|=k, \; |T \cap S| \geq 2, \,
T \neq S \bigr\} \Bigr| &\leq \sum_{\{i,j\} \subseteq S} |F_{\{i,j\}}| - 1 \\
&= \binom{k}{2} \, \binom{n-2}{k-2} - 1.
\end{align*}
Finally, by Theorem~\ref{theorem: LLL - symmetric} with
\begin{align}
\label{eq1: 30.05.26}
p = 2^{1-\binom{k}{2}}, \quad d' = \binom{k}{2} \, \binom{n-2}{k-2} - 1,
\end{align}
and $d \leq d'$, it follows that if $ep(d'+1) \leq 1$, or equivalently, if \eqref{eq0: 30.05.26} holds, then
\begin{align}
\label{eq1b: 30.05.26}
\mathbb{P}\left( \bigcap_{S \subseteq [n]: \; |S|=k} \overline{A_S} \right) > 0.
\end{align}
In other words, with positive probability, the random coloring of $K_n$ does not contain any monochromatic copy of $K_k$.
Hence $R(k,k) > n$, which proves Theorem~\ref{theorem: LB R(k,k) - ver3}.
\end{proof}

A variation of Theorem~\ref{theorem: LB R(k,k) - ver3}, which gives a slightly better lower bound on the diagonal
Ramsey numbers is given as follows.
\begin{theorem} \label{theorem: LB R(k,k) - ver4}
{\em Let $k \geq 2$ be an integer. If, for some integer $n \geq 2$,
\begin{align}
\label{eq2: 30.05.26}
e \, \Biggl[ \binom{n}{k} - \binom{n-k}{k} - k \, \binom{n-k}{k-1} \Biggr] \, 2^{1-\binom{k}{2}} \leq 1,
\end{align}
then $R(k,k) > n$.}
\end{theorem}
\begin{proof}
The proof follows the same argument as that of Theorem~\ref{theorem: LB R(k,k) - ver3},
except that we evaluate the exact value of $d$ instead of using the upper bound $d'$ in
\eqref{eq1: 30.05.26}.

For an arbitrary $k$-subset $S \subseteq [n]$, $d$ is the (fixed) number of
$k$-subsets $T \subseteq [n]$, distinct from $S$, that satisfy $|T \cap S| \geq 2$.
Since there are $\binom{n}{k}$ $k$-subsets of $[n]$, of which $\binom{n-k}{k}$
are disjoint from $S$ and $k \, \binom{n-k}{k-1}$ intersect $S$ in exactly one
element, it follows that
\begin{align}
\label{eq3: 30.05.26}
d = \binom{n}{k} - \binom{n-k}{k} - k \, \binom{n-k}{k-1} - 1.
\end{align}
Thus, applying Theorem~\ref{theorem: LLL - symmetric} with
$p = 2^{1-\binom{k}{2}}$ (see \eqref{eq1: 30.05.26})
and $d$ given by \eqref{eq3: 30.05.26}, we conclude that if
$ep \, (d+1) \leq 1$, or equivalently,
if \eqref{eq2: 30.05.26} is satisfied, then \eqref{eq1b: 30.05.26} holds.
Therefore, with positive probability, the random coloring of $K_n$ contains no
monochromatic copy of $K_k$, and so $R(k,k)>n$.
\end{proof}

Numerical experiments comparing the lower bounds on the diagonal Ramsey numbers $R(k,k)$ given in
Theorems~\ref{theorem: LB R(k,k) - ver3} and~\ref{theorem: LB R(k,k) - ver4} suggest that the
fractional improvement provided by Theorem~\ref{theorem: LB R(k,k) - ver4} decreases with $k$,
and tends to zero as $k \to \infty$.

The next result, which follows from Theorem~\ref{theorem: LB R(k,k) - ver3},
refines the asymptotic lower bound on the diagonal Ramsey numbers in \cite[Corollary~1]{Spencer75}.
\begin{proposition}  \label{prop: ramsey-asymptotic lb}
{\em For every $\varepsilon \in (0,1)$, there exists
$k_0 = k_0(\varepsilon) \in \mathbb{N}$ such that
\begin{align}
\label{eq: ramsey-asymptotic lb}
R(k,k) > (1- \varepsilon) \; \Biggl(\frac{\sqrt{2}}{e}\Biggr) \; k 2^{\frac{k}{2}}, \quad \forall \, k \geq k_0.
\end{align}}
\end{proposition}
\begin{remark}
\label{remark: k_0}
{\em An explicit closed-form expression for a valid $k_0 = k_0(\varepsilon)$ is derived in the proof
of Proposition~\ref{prop: ramsey-asymptotic lb} (see \eqref{eq: k_0}), and its behavior is then analyzed
as $\varepsilon \to 0^+$ (see \eqref{eq: approximated k_0}).}
\end{remark}
\begin{proof}
Let $c_\varepsilon := (1-\varepsilon)\,\frac{\sqrt{2}}{e}$.
We show that there exists $k_0=k_0(\varepsilon)\in \mathbb{N}$ such that
\[
R(k,k) > c_\varepsilon \, k \, 2^{k/2}, \qquad \forall \, k \geq k_0.
\]

\noindent
Fix $n := \bigl\lfloor c_\varepsilon \, k \, 2^{k/2} \bigr\rfloor$.
By Theorem~\ref{theorem: LB R(k,k) - ver3}, it suffices to show that
inequality \eqref{eq0: 30.05.26} holds for all sufficiently large $k$.
Since $n \leq c_\varepsilon k 2^{k/2}$ and
$\binom{n-2}{k-2}\leq \frac{n^{k-2}}{(k-2)!}$, we get
\begin{align}
\label{eq1: 01.06.26}
e \, \binom{k}{2} \, \binom{n-2}{k-2} \, 2^{1-\binom{k}{2}}
\leq e \, \binom{k}{2} \; \frac{(c_\varepsilon k)^{k-2}}{(k-2)!} \; 2^{1-k/2},
\end{align}
By Stirling's inequality, asserting that $n! \geq \sqrt{2\pi n} \, (n/e)^n$ for all $n \in \mathbb{N}$,
together with the inequality $\left(\frac{k}{k-2}\right)^{k-2} \leq e^2$ that holds for all $k \geq 3$,
it follows from \eqref{eq1: 01.06.26} that
\begin{align}
\label{eq1: 12.03.26}
e \, \binom{k}{2} \, \binom{n-2}{k-2} \, 2^{1-\binom{k}{2}}
\leq \sqrt{\frac{3}{8\pi}} \, e^3 \,k^{3/2} \, (1-\varepsilon)^{k-2},
\qquad \forall \, k \geq 3.
\end{align}
Since $\varepsilon \in (0,1)$, the right-hand side of \eqref{eq1: 12.03.26} tends
to zero as $k\to\infty$.
Consequently, there exists $k_0=k_0(\varepsilon) \in \mathbb{N}$ such that
\eqref{eq0: 30.05.26} holds for all $k \geq k_0$.
By Theorem~\ref{theorem: LB R(k,k) - ver3},
$R(k,k) \geq n+1$ for all $k \geq k_0$, and therefore
\begin{align*}
R(k,k) &> c_\varepsilon \, k \, 2^{k/2},
\qquad \forall \, k \geq k_0,
\end{align*}
which proves~\eqref{eq: ramsey-asymptotic lb} by the expression for $c_\varepsilon$.
\end{proof}

The proof of Proposition~\ref{prop: ramsey-asymptotic lb} enables one to get an
explicit closed-form expression for a suitable choice of $k_0 = k_0(\varepsilon)$.
Let $\varepsilon \in (0,1)$. From \eqref{eq1: 12.03.26}, together with
the requirement to satisfy \eqref{eq0: 30.05.26} for all $k \geq k_0$, it follows
that $k_0$ can be selected to be the smallest integer $k \geq 3$ that satisfies the
inequality
\begin{align}
\label{eq3: 12.03.26}
\sqrt{\frac{3}{8\pi}} \, e^3 \, k^{3/2} \, (1-\varepsilon)^{k-2} \leq 1.
\end{align}
A suitable choice of $k_0 = k_0(\varepsilon)$ is obtained by solving inequality
\eqref{eq3: 12.03.26} subject to the constraint $k \geq 3$, as detailed in Appendix~\ref{appendix: W}.
This gives
\begin{align}
\label{eq: k_0}
k_0 &= \max \Biggl\{3, \, \Biggl\lceil \frac{3 W_{-1}\bigl(\beta \ln(1-\varepsilon)
\, (1-\varepsilon)^{4/3} \bigr)}{2\ln(1-\varepsilon)} \Biggr\rceil \Biggr\},
\end{align}
where $W_{-1}(\cdot)$ denotes the secondary branch of the Lambert $W$-function, and
\begin{align}
\label{eq: beta}
\beta &= \frac{4}{3 e^2} \sqrt[3]{\frac{\pi}{3}} \approx 0.183242.
\end{align}
We have
\begin{align*}
& W_{-1}(x) = \ln(-x) - \ln\bigl(-\ln(-x)\bigr) + o(1) \qquad \text{as } x \to 0^-, \\
& \ln(1-x) = -x + O(x^2) \qquad \text{as } x \to 0,
\end{align*}
which, applied to the right-hand side of \eqref{eq: k_0}, yield
\begin{align}
\label{eq: approximated k_0}
k_0(\varepsilon) & \approx \frac{3}{2 \varepsilon} \, \biggl[ \ln \biggl( \frac{1}{\beta \varepsilon} \biggr) +
\ln \ln \biggl( \frac{1}{\beta \varepsilon} \biggr) \biggr] \\[0.1cm]
\label{eq2: approximated k_0}
& = \Theta \biggl( \frac1{\varepsilon} \, \ln \frac1{\varepsilon} \biggr).
\end{align}

We close this subsection by noting that the lower bounds on the diagonal Ramsey numbers $R(k,k)$ given in
Theorems~\ref{theorem: LB R(k,k) - ver3} and~\ref{theorem: LB R(k,k) - ver4} are well below the state-of-the-art
bounds for small values of~$k$. For example, the resulting lower bounds on $R(10,10)$ are~99 and~105, respectively,
whereas the current best lower bound is~798, due to \cite{Mathon87} (see also \cite[Table~1a]{Radziszowski2026}).
The significance of the LLL-based analysis in this subsection lies instead in Proposition~\ref{prop: ramsey-asymptotic lb}
and \eqref{eq: k_0}--\eqref{eq2: approximated k_0}, which provide a refinement of the LLL-based analysis underlying
the asymptotically best known lower bound on $R(k,k)$ due to \cite[Corollary~1]{Spencer75}.

\subsection{Lower bounds on the off-diagonal Ramsey numbers}
\label{subsection: off-diagonal Ramsey numbers}

As a continuation of Section~\ref{subsection: diagonal Ramsey numbers}, we now apply the asymmetric LLL
(Theorem~\ref{theorem: LLL}) to derive lower bounds on off-diagonal Ramsey numbers. This extends
the use of the symmetric LLL in the preceding subsection to derive lower bounds on
diagonal Ramsey numbers. This presentation follows Spencer's analysis in \cite{Spencer75,Spencer77}.

\begin{theorem}[Lower bound on off-diagonal Ramsey numbers]
\label{thm: Spencer - off-diagonal Ramsey}
{\em Let $k, \ell \geq 2$ be distinct integers, and define $a:=\binom{k}{2}$ and $b:=\binom{\ell}{2}$.
If, for some integer $n \geq 2$, there exist numbers $p, x, y \in (0,1)$ such that
\begin{align}
\label{eq2a: 01.06.26}
& p^{a} \leq x \, (1-x)^{\varepsilon} \, (1-y)^{\vartheta}, \\[0.1cm]
\label{eq2b: 01.06.26}
& (1-p)^b \leq y \, (1-x)^{\varphi} \, (1-y)^{\, \rho},
\end{align}
where
\begin{align}
\label{eq2c: 01.06.26}
\varepsilon := \binom{k}{2} \, \binom{n-2}{k-2}, \quad \vartheta := \binom{k}{2} \, \binom{n-2}{\ell-2}, 
\quad \varphi := \binom{\ell}{2} \, \binom{n-2}{k-2}, \quad \rho := \binom{\ell}{2} \, \binom{n-2}{\ell-2}.
\end{align}
Then, $R(k,\ell)>n$.}
\end{theorem}

\begin{proof}
Color the edges of the complete graph $K_n$ independently at random, assigning each edge the color blue
with probability $p$ and the color red with probability $1-p$. For a set $S \subseteq [n]$ with $|S|=k$,
let $A_{S}$ be the event that all edges with both endpoints in $S$ are colored blue; likewise, for a set
$T \subseteq [n]$ with $|T|=\ell$, let $B_{T}$ be the event that all edges with both endpoints in $T$
are colored red. Then,  $\mathbb{P}(A_S) = p^a$ and $\mathbb{P}(B_T) = (1-p)^b$. Any two of these events
are independent whenever their corresponding sets share at most one element (i.e., their corresponding
cliques share no edge).

Let $x,y \in (0,1)$, and set $x_{A_S} = x$ and $x_{B_T} = y$ for all such subsets $S, T \in [n]$ with
$|S|=k$ and $|T|=\ell$.
Consider an associated dependency graph, whose vertex set is the above family of events, and where any
two of them are adjacent if and only if they intersect in at least two elements.
\begin{enumerate}[(1)]
\item
The total number of edges linking $A_S$ with any of the events in $\{A_{S'}\}$, where $S' \in [n]$,
$|S'| = k$, and $|S \cap S'| \geq 2$, is at most $\varepsilon$. Likewise, the total number of edges linking
$A_S$ with any of the events in $\{B_{T'}\}$, where $T' \in [n], |S \cap T'| \geq 2$, and $|T'| = \ell$,
is at most $\vartheta$.
\item
The total number of edges linking $B_T$ with any of the events in $\{A_{S'}\}$, where $S' \in [n]$,
$|S'| = k$, and $|T \cap S'| \geq 2$, is at most $\varphi$. Likewise, the total number of edges
linking $B_T$ with any of the events in $\{B_{T'}\}$, where $T' \in [n], |T \cap T'| \geq 2$, and
$|T'| = \ell$, is at most $\rho$.
\end{enumerate}
By the LLL (Theorem~\ref{theorem: LLL}), it follows that if the conditions in \eqref{eq2a: 01.06.26}
and \eqref{eq2b: 01.06.26} hold, then
\begin{align}
\label{eq2f: 01.06.26}
\mathbb{P}\left( \bigcap_{S \subseteq [n]: \, |S|=k} \overline{A_S} \; \cap \; \bigcap_{T \subseteq [n]: \, |T|=\ell} \overline{B_T} \right) > 0.
\end{align}
This implies that, with positive probability, the complete graph $K_n$ contains neither a blue copy of $K_k$ nor a red copy of $K_\ell$,
which implies that $R(k, \ell) > n$.
\end{proof}

In analogy with the tightened lower bound for diagonal Ramsey numbers in Theorem~\ref{theorem: LB R(k,k) - ver4},
we next present a tightened version of Theorem~\ref{thm: Spencer - off-diagonal Ramsey}. As in the proof of
Theorem~\ref{theorem: LB R(k,k) - ver4}, the improvement is obtained by replacing the upper bounds on the numbers
of edges in each of the four cases in the same dependency graph with their exact values.

\begin{theorem}[Lower bound on off-diagonal Ramsey numbers]
\phantomsection
\label{theorem: IS - off-diagonal Ramsey}
{\em Let $k, \ell \geq 2$ be distinct integers, and define $a:=\binom{k}{2}$ and $b:=\binom{\ell}{2}$.
If, for some integer $n \geq 2$, there exist numbers $p, x, y \in (0,1)$ such that
\begin{align}
\label{eq2a: 02.06.26}
& p^{a} \leq x \, (1-x)^{\varepsilon'} \, (1-y)^{\vartheta'}, \\[0.1cm]
\label{eq2b: 02.06.26}
& (1-p)^b \leq y \, (1-x)^{\varphi'} \, (1-y)^{\, \rho'},
\end{align}
where
\begin{align}
\label{eq2c: 02.06.26}
& \varepsilon' := \binom{n}{k} - \binom{n-k}{k} - k \, \binom{n-k}{k-1} - 1,
\qquad \vartheta' := \binom{n}{\ell} - \binom{n-k}{\ell} - k \, \binom{n-k}{\ell-1}, \\[0.1cm]
\label{eq2d: 02.06.26}
& \varphi' := \binom{n}{k} - \binom{n-\ell}{k} - \ell \, \binom{n-\ell}{k-1},
\qquad \hspace*{0.5cm} \rho' := \binom{n}{\ell} - \binom{n-\ell}{\ell} - \ell \, \binom{n-\ell}{\ell-1} - 1.
\end{align}
Then, $R(k,\ell)>n$.}
\end{theorem}
\begin{remark}
{\em Following the proof of Theorem~\ref{thm: Spencer - off-diagonal Ramsey}, the subtraction of~1 in the definitions 
of $\varepsilon'$ and $\rho'$ in \eqref{eq2c: 02.06.26} and \eqref{eq2d: 02.06.26}, respectively, accounts for 
the fact that, in the corresponding same-type counts, the event itself must be excluded. In contrast, no subtraction 
of~1 appears in the mixed counts corresponding to $\vartheta'$ and $\varphi'$, since the events $A_S$ and $B_T$ 
are distinct for all $S$ and $T$ with $|S|=k$ and $|T|=\ell$ (even when $S=T$ and $k=\ell$).}
\end{remark}

The next asymptotic lower bound on off-diagonal Ramsey numbers is a consequence of
Theorem~\ref{thm: Spencer - off-diagonal Ramsey} (see \cite[Theorem~2.2]{Spencer77} for a proof).
\begin{theorem}[Asymptotics]
\label{thm: Spencer - off-diagonal Ramsey - asymptotic}
{\em For every fixed integer $k \geq 3$, there exists a constant $c_k>0$ such that
\begin{align}
\label{eq:LLL - off-diagonal Ramsey - asymptotic}
R(k,\ell) \geq c_k \left( \frac{\ell}{\log \ell} \right)^{\frac{k+1}{2}}
\end{align}
for all sufficiently large $\ell$.}
\end{theorem}

We close this subsection by noting that several recent works (see \cite{MaSX2026, HunterMS2025, LinN2026, Bradac2026})
derive improved lower bounds on off-diagonal Ramsey numbers using probabilistic techniques that do not rely on the LLL.

\subsection{Coloring hypergraphs}
\label{subsection: Coloring hypergraphs}

Using the LLL, the seminal work \cite{LLL75} established the existence of hypergraphs with
a prescribed chromatic number. Their approach highlights the power of probabilistic methods in demonstrating the
existence of sparse hypergraphs with nontrivial coloring properties.

This subsection concerns the coloring of hypergraphs and is based in part on \cite{LLL75} and \cite[Section~5.2]{AlonSpencer2016},
together with some reformulated proofs and strengthened results on hypergraph colorings (Theorems~\ref{theorem1: Erdos-Lovasz 75},
\ref{theorem2: Erdos-Lovasz 75}, \ref{theorem3: sharpened Erdos-Lovasz 75}, and Remark~\ref{remark: sharpened Erdos-Lovasz 75}).

\begin{definition}
{\em A \emph{hypergraph} $H = (V,E)$ consists of a set of vertices $V$ and a collection $E$ of subsets of $V$.
Each element $h \in E$ is called a \emph{hyperedge}. A hypergraph $H=(V,E)$ is \emph{$k$-colorable} if there is
a $k$-coloring of $V$ such that no hyperedge is monochromatic.}
\end{definition}

\begin{theorem} \label{theorem: k-colorable hypergraph}
{\em Let $H = (V,E)$ be a hypergraph in which every hyperedge has size at least $r$, and each hyperedge
intersects at most $d$ other hyperedges. If $ed \leq k^{r-1}$, with $k \geq 2$, then $H$ is $k$-colorable.}
\end{theorem}
\begin{proof}
Assign to each vertex $v$ of $H$ a random color chosen independently and uniformly from $[k]$. For each hyperedge
$h \in E$, let $A_h$ be the bad event that $h$ is monochromatic. By assumption $|h| \geq r$, which implies that
\begin{align}
\label{eq1: 31.05.26}
\mathbb{P}(A_h) = k \left(\frac{1}{k}\right)^{|h|} \leq k^{1-r} \qquad \text{for all } h \in E.
\end{align}
Moreover, due to the independent coloring of the vertices,
$A_h$ is independent of the $\sigma$-algebra generated by all bad events $A_{h'}$ such that $h \cap h' = \varnothing$.
Since $h$ intersects at most $d$ other hyperedges, it follows that $A_h$ is independent of the $\sigma$-algebra generated by all
other bad events except for at most $d$ of them.
Hence, there exists a dependency graph for the family of events $\{A_h\}_{h\in E}$ whose maximum degree is at most $d$.
Applying Corollary~\ref{corollary: Knuth} with $p = k^{1-r}$, it follows that if $e d k^{1-r} \leq 1$, then
\begin{align}
\label{eq2: 31.05.26}
\mathbb{P}\left( \bigcap_{h \in E} \overline{A_h} \right) > 0.
\end{align}
This shows that if $ed \leq k^{r-1}$, then $H$ admits a $k$-coloring in which no hyperedge is monochromatic. By definition, $H$ is $k$-colorable.
\end{proof}

\begin{definition}
{\em A hypergraph $H$ is \emph{$r$-uniform} if each hyperedge has size $r$, and it is
\emph{$b$-regular} if every vertex in $V$ is contained in exactly $b$ hyperedges of $H$.}
\end{definition}

\begin{corollary} \label{cor1: 2-colorable}
{\em Let $H$ be an $r$-uniform and $b$-regular hypergraph with $b,r \geq 2$, and let $k \geq 2$.
If $er(b-1) \leq k^{r-1}$, then $H$ is $k$-colorable.}
\end{corollary}
\begin{proof}
Each hyperedge $h \in E$ contains exactly $r$ vertices, and each of these vertices belongs to exactly $b$ hyperedges.
Therefore, for every $h \in E$, the number of hyperedges intersecting $h$ is at most $r(b-1)$.
Hence, Theorem~\ref{theorem: k-colorable hypergraph} applies with $d=r(b-1)$, yielding the desired result.
\end{proof}

\begin{corollary} \label{cor2: 2-colorable}
{\em Let $H$ be an $r$-uniform and $r$-regular hypergraph. Then $H$ is 2-colorable if $r \geq 9$, 3-colorable if $r \geq 5$,
4-colorable if $r \geq 4$, 5-colorable if $r \geq 3$, and 6-colorable if $r \geq 2$.}
\end{corollary}
\begin{proof}
By Corollary~\ref{cor1: 2-colorable}, applied with $r=b$, it follows that if $er(r-1) \leq k^{r-1}$, then
$H$ is $k$-colorable. This inequality holds if and only if $r \geq 9, 5, 4, 3$, and $2$ for $k = 2, 3, 4, 5$,
and $6$, respectively.
\end{proof}

\smallskip
The next two results restate \cite[Theorem~2]{LLL75} and \cite[Theorem~3]{LLL75}, respectively, and present reformulated
proofs that yield slightly stronger conclusions, based on Corollary~\ref{corollary: Knuth} and on replacing the constant~$4$
in \cite{LLL75} by~$e$.
Moreover, these proofs are somewhat simpler, as they do not rely on the line graph of $H$, but instead work directly with $H$.
\begin{theorem}
\label{theorem1: Erdos-Lovasz 75}
{\em Let $k, r \geq 2$, and let $H$ be an $r$-uniform hypergraph.
If every hyperedge of $H$ intersects at most $\frac{1}{e} \, k^{r-1}$ other hyperedges, then $H$ is $k$-colorable.
In particular, if every vertex of $H$ has degree at most $\frac{k^{r-1}}{er}$, then $H$ is $k$-colorable.}
\end{theorem}
\begin{proof}
The first assertion is a direct restatement of Theorem~\ref{theorem: k-colorable hypergraph}. Since every hyperedge of $H$
intersects at most $\frac{1}{e}\,k^{r-1}$ other hyperedges, it follows that $H$ is $k$-colorable.

The second assertion follows immediately from the first.
Indeed, every hyperedge contains $r$ vertices, each of degree at most
$\frac{k^{r-1}}{er}$.
Therefore, the number of hyperedges intersecting a given hyperedge is at most
$r\cdot \frac{k^{r-1}}{er} = \frac{1}{e}\,k^{r-1}$,
where we rely on the subadditivity of cardinality.
Applying the first assertion completes the proof.
\end{proof}

\begin{theorem}
\label{theorem2: Erdos-Lovasz 75}
{\em Let $r \geq k \geq 2$, and let $H$ be an $r$-uniform hypergraph. If every hyperedge of $H$ intersects at most
$\left\lfloor \frac{k^{r-1}}{e (k-1)^r} \right\rfloor$ other hyperedges, then $H$ admits a $k$-coloring in which
each hyperedge contains all colors.}
\end{theorem}
\begin{proof}
As in the proof of Theorem~\ref{theorem: k-colorable hypergraph}, assign to each vertex
of $H$ a color chosen independently and uniformly at random from $[k]$.
For each hyperedge $h \in E$, let $A_h$ denote the bad event that $h$ does not contain all $k$ colors.
Since the event $A_h$ is the union of the $k$ events that a particular color is missing from $h$, it
follows by the union bound that
\begin{align}
\label{eq1: 23.04.26}
\mathbb{P}(A_h) \leq k \left(1-\frac1k\right)^r := p \qquad \text{for all } h \in E.
\end{align}
Since $A_h$ depends only on the colors assigned to the vertices of $h$, and the vertices of $H$ are colored independently,
$A_h$ is independent of the $\sigma$-algebra generated by all bad events $A_{h'}$ such that $h \cap h' = \varnothing$.
Let
\begin{align}
\label{eq3: 31.05.26}
d:= \left\lfloor \frac{k^{r-1}}{e (k-1)^r} \right\rfloor.
\end{align}
Since $h$ intersects at most $d$ other hyperedges, it follows that
$A_h$ is independent of the $\sigma$-algebra generated by all other bad events except for at most $d$ of them.
Hence, there exists a dependency graph for the family of events $\{A_h\}_{h\in E}$ with maximum degree at most $d$.
This implies that
\begin{align}
\label{eq4: 31.05.26}
epd \leq e k \left(1-\frac1k\right)^r \cdot \frac{k^{r-1}}{e (k-1)^r} = 1.
\end{align}
By Corollary~\ref{corollary: Knuth}, with positive probability, no
bad event $A_h$ ($h \in E$) occurs. Therefore, there exists a $k$-coloring of the vertices of $H$ in which every
hyperedge contains all $k$ colors.
\end{proof}

\noindent
In the following, we propose a stronger version of Theorem~\ref{theorem2: Erdos-Lovasz 75}.
To derive an exact closed-form expression for the probability $\mathbb{P}(A_h)$ in place
of the upper bound on the right-hand side of \eqref{eq1: 23.04.26}, we recall the definition
of the Stirling numbers of the second kind.

\begin{definition}
\label{def: Stirling number of 2nd kind}
{\em Let $n, k \in \mathbb{N}$. The {\em Stirling number of the second kind}, denoted by $S(n,k)$, is
the number of ways to partition the set $[n]$ into $k$ nonempty and pairwise disjoint subsets. If
$k > n$, then $S(n,k)=0$.}
\end{definition}

In the proof of Theorem~\ref{theorem2: Erdos-Lovasz 75}, the $r$ vertices of a hyperedge $h \in E$
are independently assigned colors chosen uniformly at random from $[k]$. Hence, all $k^r$ colorings
of $h$ are equally likely. By Definition~\ref{def: Stirling number of 2nd kind}, the number of colorings
of $h$ that use all $k$ colors is $k!\,S(r,k)$. Therefore, the probability of the bad event $A_h$,
namely, that $h$ does not contain all $k$ colors, is
\begin{align}
\label{eq2: 23.04.26}
\mathbb{P}(A_h) = 1 - \frac{k! \, S(r,k)}{k^r}.
\end{align}
By \cite[Eq.~(6.19)]{GrahamKP89}, which gives a closed-form expression for $S(r,k)$, it follows that if $r \geq k \geq 2$, then
\begin{align}
\label{eq3: 23.04.26}
\mathbb{P}(A_h) = 1 - \frac1{k^r} \sum_{j=0}^k \Biggl\{ (-1)^{k-j} \, \binom{k}{j} \, j^{\, r} \Biggr\}, \qquad h \in E.
\end{align}
This allows one to sharpen the result of Theorem~\ref{theorem2: Erdos-Lovasz 75} through an application of Corollary~\ref{corollary: Knuth},
replacing the upper bound $p$ on $\mathbb{P}(A_h)$ in \eqref{eq1: 23.04.26} with the exact value given by \eqref{eq3: 23.04.26}. Consequently,
the following result holds:
\begin{theorem}
\label{theorem3: sharpened Erdos-Lovasz 75}
{\em Let $r \geq k \geq 2$, and let $H$ be an $r$-uniform hypergraph. If every hyperedge of $H$ intersects at most
\[
\frac1{e \, \Biggl( 1 - \frac1{k^r} \overset{k}{\underset{j=0}{\sum}} \Bigl\{ (-1)^{k-j} \, \binom{k}{j} \, j^{\, r} \Bigr\} \Biggr)}
\]
other hyperedges, then $H$ admits a $k$-coloring in which each hyperedge contains all colors.}
\end{theorem}

The following remark shows that the improvement afforded by the sharpened result in Theorem~\ref{theorem3: sharpened Erdos-Lovasz 75}
over Theorem~\ref{theorem2: Erdos-Lovasz 75} allows one to increase the number of hyperedges intersecting any given hyperedge by at most~1.
\begin{remark}
\label{remark: sharpened Erdos-Lovasz 75}
{\em Let $r \geq k \geq 2$, and define
\begin{align}
D_{r,k} := k\left(1 - \frac{1}{k}\right)^r,
\qquad
N_{r,k} := 1 - \frac{k! \, S(r,k)}{k^r}.
\end{align}
Furthermore, set
\begin{align}
\label{eq5: 31.05.26}
A_{r,k} := \left\lfloor \frac{1}{e D_{r,k}} \right\rfloor,
\qquad
B_{r,k} := \left\lfloor \frac{1}{e N_{r,k}} \right\rfloor.
\end{align}
Then $A_{r,k}$ and $B_{r,k}$ represent, respectively, the largest number of hyperedges that may intersect a given hyperedge
in an $r$-uniform hypergraph such that Theorem~\ref{theorem2: Erdos-Lovasz 75} or Theorem~\ref{theorem3: sharpened Erdos-Lovasz 75}
guarantees the existence of a $k$-coloring in which every hyperedge contains all $k$ colors.

\noindent
The quantity $D_{r,k}$ is the expected number of empty boxes when $r$ balls are thrown independently and uniformly
into $k$ boxes, while $N_{r,k}$ is the probability that at least one box is empty. Indeed, $D_{r,k}$ is the sum of the
probabilities that the individual boxes are empty, and therefore, by the union bound, $N_{r,k} \leq D_{r,k}$. Consequently,
$B_{r,k} \geq A_{r,k}$, and we next show that their difference is at most one.
Let $X$ denote the number of empty boxes, i.e.,
\[
X = \sum_{i=1}^k I_i \, ,   \qquad
I_i := \mathbf{1}_{\{ \text{box } i \text{ is empty} \}} \quad \text{for all } i \in [k],
\]
where $\mathbf{1}_{E}$ denotes the indicator function of an event $E$. Then,
\begin{align*}
& \mathbb{E}[I_i] = \left(1-\frac{1}{k}\right)^r, \qquad \mathbb{E}[X] = D_{r,k},
\qquad \mathbb{P}(X \geq 1) = N_{r,k},  \\
& \mathbb{E}[I_i I_j] = \left(1-\frac{2}{k}\right)^r, \quad \text{for all } i \neq j,
\end{align*}
so
\begin{align*}
\mathrm{Cov}(I_i,I_j) &= \left(1-\frac{2}{k}\right)^r - \left(1-\frac{1}{k}\right)^{2r} < 0,
\end{align*}
thus the indicators $\{I_i\}$ are negatively correlated. Therefore,
\[
\mathrm{Var}(X) = \sum_{i=1}^k \mathrm{Var}(I_i) + \sum_{i \ne j} \mathrm{Cov}(I_i,I_j)
\leq \sum_{i=1}^k \mathrm{Var}(I_i).
\]
Since each $I_i$ is a Bernoulli random variable with mean $p=\left(1-\frac{1}{k}\right)^r$, it follows that
\begin{align*}
& \mathrm{Var}(I_i)=p(1-p)\leq p, \\[0.1cm] 
& \mathrm{Var}(X) \leq k p = D_{r,k}.
\end{align*}
By the Cauchy--Schwarz inequality, and since $X$ is a nonnegative integer-valued random variable,
\begin{align}
(\mathbb{E}[X])^2 &= (\mathbb{E}[X \mathbf{1}_{\{X \geq 1\}}])^2 \nonumber \\
&\leq \mathbb{E}[X^2] \; \mathbb{E}[\mathbf{1}_{\{X \geq 1\}}] \nonumber \\
\label{eq1: 17.06.2026}
&= \mathbb{E}[X^2] \; \mathbb{P}(X \geq 1), 
\end{align}
and therefore 
\begin{align}
\label{eq6: 31.05.26}
N_{r,k} = \mathbb{P}(X \geq 1) \geq \frac{(\mathbb{E}[X])^2}{\mathbb{E}[X^2]}.
\end{align}
Since
\[
\mathbb{E}[X^2] = \mathrm{Var}(X)+(\mathbb{E}[X])^2 \leq D_{r,k}+D_{r,k}^2,
\]
we get
\[
N_{r,k} \geq \frac{D_{r,k}^2}{D_{r,k}+D_{r,k}^2} = \frac{D_{r,k}}{1+D_{r,k}}.
\]
Consequently,
\[
0 \leq \frac{1}{e N_{r,k}} - \frac{1}{e D_{r,k}} \leq \frac{1}{e} < 1,
\]
and therefore, by \eqref{eq5: 31.05.26},
\begin{align}
B_{r,k} - A_{r,k} \in \{0,1\}.
\end{align}
The equality $B_{r,k} = A_{r,k} + 1$ can hold for some pairs $(r,k)$ with $r \geq k \geq 2$, and
Table~\ref{table2} provides several such numerical examples.
\begin{table}[hbt]
\centering
\begin{tabular}{|c||c|c|}
\hline
$(r,k)$ & $A_{r,k}$ & $B_{r,k}$ \\
\hline
(21, 5) & 7 & 8 \\
(22, 5) & 9 & 10 \\
(19, 6) & 1 & 2 \\
(28, 7) & 3 & 4 \\
(35, 8) & 4 & 5 \\
(41, 8) & 10 & 11 \\
(48, 8) & 27 & 28 \\
\hline
\end{tabular}
\vspace*{0.2cm}
\caption{\centering{Examples of pairs $(r,k)$ with $r \geq k \geq 2$ for which the equality $B_{r,k} = A_{r,k} + 1$ holds.}}
\label{table2}
\end{table}}
\end{remark}

\subsection{Length of cycles in directed graphs}
\label{subsection: Length of cycles in directed graphs}

The following result, due to Alon and Linial \cite{AlonL89}, provides an insightful
application of the symmetric version of the LLL (Theorem~\ref{theorem: LLL - symmetric}).
Unlike the three preceding applications in Sections~\ref{subsection: diagonal Ramsey numbers},
\ref{subsection: off-diagonal Ramsey numbers}, and~\ref{subsection: Coloring hypergraphs},
the next theorem is not directly concerned with graph coloring, although its proof relies
on coloring arguments.

\begin{theorem}[\hspace*{-0.1cm} \cite{AlonL89}]
\label{theorem: Alon and Linial, 1989}
{\em Let $D = (V,E)$ be a finite, simple, and directed graph with minimum out-degree $\delta \geq 1$
and maximum in-degree at most $\Delta \geq 1$. If, for some $k \in \mathbb{N}$,
\begin{align}
\label{eq: Alon and Linial, 1989}
e \bigl( \delta \Delta + 1 \bigr) \, \biggl(1 - \frac1k \biggr)^{\delta} \leq 1,
\end{align}
then $D$ contains a directed cycle whose length is divisible by $k$.}
\end{theorem}
\begin{remark}
\label{remark: Alon and Linial, 1989 - relaxed}
{\em By Corollary~\ref{corollary: Knuth} and the proof of Theorem~\ref{theorem: Alon and Linial, 1989},
the condition in \eqref{eq: Alon and Linial, 1989} can be relaxed to
\begin{align}
\label{eq: Alon and Linial, 1989 - relaxed}
e \delta \Delta \, \biggl(1 - \frac1k \biggr)^{\delta} \leq 1.
\end{align}}
\end{remark}

\noindent
For completeness, we present a more detailed version of the proof appearing in \cite[p.~88]{AlonSpencer2016}
and \cite[Theorem~2.3]{AlonL89}.
\begin{proof}
We may assume that all out-degrees are equal to $\delta$. Indeed, by deleting arcs if
necessary, we obtain a spanning subgraph of $D$ in which every vertex has out-degree
$\delta$, while the maximum in-degree remains at most $\Delta$.
Any directed cycle in this spanning subgraph is also present in $D$; in particular,
so is any directed cycle whose length is divisible by $k$.

We use the LLL to show that there exists a vertex coloring
$c \colon V \to [k]$ such that every vertex $v \in V$ has an out-neighbor $u$
with $c(u) \equiv c(v) + 1 \pmod{k}$.

We first show that the existence of such a $k$-coloring implies the theorem. Let
$c \colon V \to [k]$ be such a coloring. Since $D$ is finite and every
vertex has out-degree at least~1, we may construct a directed walk as follows.
Start at some vertex $v_1 \in V$, and for each $i \geq 1$, choose an out-neighbor
$v_{i+1}$ of $v_i$ with $c(v_{i+1}) \equiv c(v_i) + 1 \pmod{k}$. Since $D$
is finite, some vertex must repeat. Let $j$ be the smallest index such that
$v_j = v_\ell$ for some $\ell < j$. Consider the directed cycle
\[
v_\ell \to v_{\ell+1} \to \cdots \to v_j = v_\ell.
\]
Then,
\begin{align*}
c(v_j) \equiv c(v_{j-1}) + 1 \equiv \cdots \equiv c(v_\ell) + (j-\ell) \pmod{k},
\end{align*}
and therefore
\[
(j-\ell) \equiv 0 \pmod{k}.
\]
Hence, the constructed directed cycle in $D$ has length $j-\ell$, which is divisible by $k$.

Next, we use the LLL to prove the existence of such a $k$-coloring.
Let $N_{+}(v)$ and $N_{-}(v)$ denote the sets of out-neighbors and in-neighbors of a vertex $v$ in
a directed graph $D$, respectively. These are referred to as the \emph{open out-neighborhood} and
\emph{open in-neighborhood} of $v$. Furthermore, define
\[
N_{+}[v] := N_{+}(v) \cup \{v\}, \qquad
N_{-}[v] := N_{-}(v) \cup \{v\},
\]
which are the \emph{closed out-neighborhood} and \emph{closed in-neighborhood} of $v$, respectively.

We next show that there exists a coloring $c \colon V \to [k]$ such that every vertex $v \in V$ has
an out-neighbor $u \in N_{+}(v)$ satisfying $c(u) \equiv c(v) + 1 \pmod{k}$.
Consider a random coloring $c \colon V \to [k]$, where each vertex of $D$ is assigned a color
independently and uniformly at random (each color with probability $\frac1k$).
For each $v \in V$, let $A_v$ denote the event that $c(u) \not\equiv c(v) + 1 \pmod{k}$ for all
$u \in N_{+}(v)$. We show, using Theorem~\ref{theorem: LLL - symmetric}, that
\begin{align}
\label{eq: 16.3.26}
\mathbb{P}\left( \bigcap_{v \in V} \overline{A_v} \right) > 0,
\end{align}
which gives the desired result. For each $v \in V$, since the colors are assigned independently and
uniformly, we have
\begin{align}
\label{eq1: 15.04.26}
\mathbb{P}(A_v) = \biggl(1 - \frac{1}{k} \biggr)^{\delta},
\end{align}
where, by assumption, $\delta = \lvert N_{+}(v) \rvert$ for all $v \in V$. Moreover, for every $v \in V$,
if the colors of all vertices outside $N_{+}(v)$ are fixed, then the conditional probability of $A_v$
remains $\biggl(1 - \frac{1}{k} \biggr)^{\delta}$.
It therefore follows that $A_v$ is mutually independent of the collection
\begin{align}
\label{eq: cond. on neighborhoods}
\bigl\{ A_u : N_{+}(v) \cap N_{+}[u] = \varnothing \bigr\}.
\end{align}
The condition in \eqref{eq: cond. on neighborhoods} fails if and only if one of the following holds:
\begin{enumerate}
\item $u \in N_{+}(v)$,
\item $N_{+}(u) \cap N_{+}(v) \neq \varnothing$ and $u \neq v$.
\end{enumerate}
The number of vertices $u$ satisfying Condition~1 is $\delta$.
Also, for each $z \in N_{+}(v)$, there are at most $\Delta-1$ vertices $u \neq v$ such that
$u \in N_{-}(z)$. Since there are $\delta$ such vertices $z$, the number of vertices satisfying
Condition~2 is at most $\delta(\Delta-1)$. Therefore, the number of vertices $u \in V$ for which
$N_{+}(v) \cap N_{+}[u] \neq \varnothing$ is at most
\[
\delta+\delta(\Delta-1)=\delta\Delta.
\]
Consequently, by the symmetric version of the LLL (Theorem~\ref{theorem: LLL - symmetric})
with parameters
\[
p = \biggl(1 - \frac{1}{k} \biggr)^{\delta}, \qquad d = \delta \Delta,
\]
it follows that \eqref{eq: 16.3.26} holds provided that \eqref{eq: Alon and Linial, 1989} is satisfied,
which completes the proof.
\end{proof}

A \emph{$d$-regular digraph} is a directed graph in which the in-degree and out-degree
of every vertex are equal to $d$. The next corollary is a general result that includes 
regular digraphs as a special case, whereas the subsequent corollary is stated exclusively 
for regular digraphs.
\begin{corollary}
\label{cor: even-length directed cycles}
{\em Let $D = (V,E)$ be a finite, simple, and directed graph with minimum out-degree
$\delta \geq 1$ and maximum in-degree at most $\Delta \geq 1$. If
\begin{align}
\label{eq: even-length directed cycles}
\Delta \leq \frac{2^\delta}{e \delta},
\end{align}
then the digraph $D$ contains an even-length directed cycle. In particular, this holds
for every $d$-regular digraph with $d \geq 8$.}
\end{corollary}
\begin{proof}
The first part follows from Theorem~\ref{theorem: Alon and Linial, 1989}
and Remark~\ref{remark: Alon and Linial, 1989 - relaxed}  with $k=2$. The second part
holds since, for a $d$-regular digraph, $\delta = \Delta = d$, and
\[
\frac{2^x}{e x^2} \geq 1 \quad \iff \quad x \geq 7.09719 \ldots \, ,
\]
so the condition in \eqref{eq: even-length directed cycles} holds for every $d$-regular
digraph if and only if $d \geq 8$.
\end{proof}

\noindent 
The next result is an improved version of \cite[Corollary~2.5]{AlonSpencer2016}.
\begin{corollary}
\label{cor: cycles divisible by k in regular digraphs}
{\em For every finite, simple, and $d$-regular digraph $D$ with $d \geq 1$, and for every integer $k$ satisfying
\begin{align}
\label{eq1: 12.06.26}
k \leq \frac{\sqrt[d]{ed^2}}{\sqrt[d]{ed^2}-1},
\end{align}
there exists a directed cycle in $D$ whose length is divisible by $k$.}
\end{corollary}
\begin{proof}
For a $d$-regular digraph, we have $\Delta = \delta = d$.
By Theorem~\ref{theorem: Alon and Linial, 1989} and Remark~\ref{remark: Alon and Linial, 1989 - relaxed},
there exists a directed cycle of length divisible by $k$ if
\[
e d^2 \, \biggl(1-\frac{1}{k} \biggr)^{d} \leq 1.
\]
Taking the $d$-th root of both sides and rearranging terms gives the condition in \eqref{eq1: 12.06.26}.
\end{proof}
\begin{remark}
\label{remark: strengthened result - 12.06.26}
{\em Since the digraph $D$ is, by assumption, finite and every vertex in $D$ has an out-degree at least~1, 
it follows that $D$ contains a directed cycle.   
Corollary~\ref{cor: cycles divisible by k in regular digraphs} is therefore nontrivial if and only if the right-hand side of 
\eqref{eq1: 12.06.26} is at least~2; this occurs if and only if $d \geq 8$. Moreover, it strengthens 
\cite[Corollary~2.5]{AlonSpencer2016}, which asserts the same conclusion under the stronger condition
\begin{align}
\label{eq2: 12.06.26}
k \leq \frac{d}{1+\ln(d^2+1)},
\end{align}
whose right-hand side is strictly smaller than that of \eqref{eq1: 12.06.26}.
For comparison, the right-hand side of \eqref{eq2: 12.06.26} is at least~2 if and only if $d \geq 12$.
It can be verified that the difference between the right-hand sides of \eqref{eq1: 12.06.26} and \eqref{eq2: 12.06.26} 
converges to $\tfrac12$ as $d \to \infty$. Indeed, let $x_d := \frac{1+2 \ln d}{d}$, then $\sqrt[d]{ed^2} = e^{x_d}$ 
and $x_d \to 0$ as we let $d \to \infty$. Hence, the right-hand side of \eqref{eq1: 12.06.26} is 
\[
\frac{\sqrt[d]{ed^2}}{\sqrt[d]{ed^2}-1} = \frac{e^{x_d}}{e^{x_d}-1} = \frac{1}{1 - e^{-x_d}} = \frac{1}{x_d} + \frac12 + o(1),
\]
and the right-hand side of \eqref{eq2: 12.06.26} is 
\[
\frac{d}{1+\ln(d^2+1)} = \frac{d}{1 + 2 \ln d + \ln (1+d^{-2})} = \frac{1}{x_d} + o(1). 
\]
Subtracting yields 
\begin{align}
\label{eq3: 12.06.26}
\lim_{d \to \infty} \, \Biggl(\frac{\sqrt[d]{ed^2}}{\sqrt[d]{ed^2}-1} - \frac{d}{1+\ln(d^2+1)} \Biggr) = \frac12. 
\end{align} 
As an illustration of the improvement in Corollary~\ref{cor: cycles divisible by k in regular digraphs}, if $d \in \{17, 18, 19, 20, 21\}$, 
then \eqref{eq1: 12.06.26} allows $k \in \{2,3\}$, whereas \eqref{eq2: 12.06.26} permits only $k=2$. Consequently, 
every finite, simple, and $d$-regular digraph with $17 \leq d \leq 21$ not only contains a directed cycle of even length, 
but also contains a directed cycle whose length is divisible by~3.}
\end{remark}

We conclude this section by quoting a result on directed cycles in regular digraphs from
\cite[Theorem~2]{AlonMM1996}, whose proof relies on the symmetric version of the LLL
(Theorem~\ref{theorem: LLL - symmetric}).
\begin{theorem}[\hspace*{-0.1cm} \cite{AlonMM1996}]
\label{theorem: Alon, McDiarmid, and M. Molloy, 1996}
{\em Let $G$ be a $d$-regular digraph without parallel edges. Then $G$ contains at least 
$\varepsilon d^2$ pairwise edge-disjoint directed cycles for some absolute constant $\varepsilon > 0$.
In particular, $\varepsilon = \frac{3}{2^{19}}$ suffices.}
\end{theorem}

\subsection{Acyclic coloring of graphs}
\label{subsection: Acyclic coloring of graphs}

This last subsection presents a sophisticated application of the LLL (Theorem~\ref{theorem: LLL}),
in the context of graph coloring, appearing in a paper by Alon, McDiarmid, and Reed \cite{AlonMR91}.

Let $G=(V,E)$ be a finite, simple, and undirected graph. 
For $v \in V$, let $N(v) = \bigl\{ u \in V: \{u, v\} \in E \bigr\}$ denote the neighborhood of $v$, 
let $d(v) = |N(v)|$ denote its degree, and let $\Delta =\Delta(G) := \underset{v \in V}{\max} \, d(v)$ 
denote the maximum degree of $G$. For vertices $u, v \in V$, let $\lambda(u,v) := |N(u) \cap N(v)|$
denote the number of common neighbors of $u$ and $v$ in $G$. Recall that a $k$-coloring $f \colon V \to [k]$ 
is called \emph{proper} if $\{u,v\} \in E$ implies that $f(u) \neq f(v)$ (i.e., no two adjacent vertices 
are assigned the same color). A coloring of $G$ is called \emph{acyclic} if it is proper and contains no 
bichromatic cycle (equivalently, every cycle in $G$ receives at least three colors). The minimum number of 
colors required for an acyclic coloring, denoted by $A(G)$, is called the \emph{acyclic chromatic number} 
of $G$. Clearly, $A(G) \geq \chi(G)$, where $\chi(G)$ denotes the chromatic number of $G$, i.e., the minimum 
number of colors required for a proper coloring. Equality holds, for example, if $G$ is an odd cycle, a 
complete graph, or a forest.

Introduced by Gr{\"u}nbaum in the study of planar graphs \cite{Grunbaum73}, the acyclic chromatic number has become an
important topic in graph coloring and structural graph theory. Borodin settled Gr{\"u}nbaum's conjecture by proving that
$A(G) \leq 5$ for every planar graph $G$ \cite{Borodin79}. A major reason for the importance of acyclic colorings is their
close connection to sparse graph classes. These connections have led to numerous applications in structural and algorithmic
graph theory, including graph decompositions, graph homomorphisms, and efficient algorithms for sparse graphs. Acyclic
colorings also play a role in the study of graph layouts, graph drawing, and graphs embedded on surfaces (see, e.g.,
\cite{NesetrilM2012}).

In \cite{AlonMR91}, it is proved that there exists an absolute positive constant $C$ such that every graph
of maximum degree $\Delta$ admits an acyclic coloring with at most $C \Delta^{4/3}$ colors (in \cite[Theorem~1.1]{AlonMR91},
$C=50$, though no attempt was made to optimize that constant). The main tool in proving this result in \cite{AlonMR91}
is the LLL (Theorem~\ref{theorem: LLL}), and a key innovation in \cite{AlonMR91} is the introduction of special
pairs of vertices and a carefully chosen family of bad events that allow the LLL to yield an upper bound of order $\Delta^{4/3}$.
This order of growth in $\Delta$ is shown to be nearly optimal. More concretely, letting
$B(\Delta) := \sup \bigl\{ A(G): \Delta(G) = \Delta \bigr\}$, for $\Delta \in \mathbb{N}$, \cite[Theorem~1.2]{AlonMR91}
establishes the asymptotic lower bound
\begin{align}
\label{eq: asymptotic LB of A(G)}
B(\Delta) = \Omega \left(\frac{\Delta^{\frac{4}{3}}}{\sqrt[3]{\log \Delta}} \right).
\end{align}
Following standard notation, let the complete bipartite graph with vertex classes of sizes $a$ and $b$ be denoted by $K_{a,b}$.
The proof of \eqref{eq: asymptotic LB of A(G)} relies on a probabilistic construction that gives, almost surely, a
graph $G$ with maximum degree $\Delta$, containing no copy of $K_{2, \gamma}$, where $\gamma = O(\Delta^{2/3} \, \sqrt[3]{\log \Delta})$,
and whose acyclic chromatic number $A(G)$ scales like $\Delta^{4/3} \bigl(\log \Delta \bigr)^{-1/3}$. This contrasts with planar graphs,
which may have arbitrarily large maximum degree, yet whose acyclic chromatic number is at most~5.

The next result is derived in \cite{AlonMR91} via an application of the LLL, and the presentation here follows
the main ideas of the proof.
\begin{theorem}[Theorem~1.1 of \cite{AlonMR91}]
\label{theorem: AlonMR91}
{\em For every graph finite, simple, and undirected graph $G$ of maximum degree $\Delta$, its acyclic chromatic number satisfies
\begin{align}
\label{eq: AlonMR91}
A(G) \leq \lceil 50 \, \Delta^{\frac{4}{3}} \rceil.
\end{align}}
\end{theorem}
\begin{proof}
Let each vertex in $G$ be assigned a color from $[k]$ independently and uniformly at random, where
$k := \lceil 50 \, \Delta^{4/3} \rceil$.
The central idea is to distinguish pairs of nonadjacent vertices having many common neighbors. A pair of nonadjacent vertices,
$\{u,v\} \notin E$, is called \emph{a special pair} if $\lambda(u,v) > \Delta^{2/3}$.
\begin{lemma}
\label{lemma: special pairs}
{\em For every vertex $u \in V$,
\begin{align}
\label{eq: special pairs}
\bigl\vert \{v:\{u,v\}\text{ is a special pair in $G$}\} \bigr\vert < \Delta^{\frac{4}{3}}.
\end{align}}
\end{lemma}
\begin{proof}
The number of all length-$2$ paths in $G$ is at most $\Delta (\Delta-1) < \Delta^2$, and every special pair contributes
more than $\Delta^{2/3}$ such paths. Let $u \in V$ and let $s=s(u)$ denote the number of special partners of $u$.
Since each special partner of $u$ contributes more than $\Delta^{2/3}$ paths of length~$2$ having $u$ as an endpoint,
we have $s \Delta^{2/3} < \Delta^2$, and therefore $s < \Delta^{4/3}$.
\end{proof}
\noindent
To apply the LLL, four types of bad events are introduced for such a coloring $f \colon V \to [k]$.
\begin{enumerate}[(1)]
\item Monochromatic edges: Let $e = \{u,v\} \in E$. Define $A = A_e := \bigl\{ f(u)=f(v)\}$, whose probability is
$\mathbb{P}(A) = \frac1k$.
\item Bichromatic induced paths of length four: Let $u_1-u_2-u_3-u_4-u_5$ be an induced $4$-path. Define
$B := \bigl\{f(u_1) = f(u_3) = f(u_5), \; f(u_2) = f(u_4) \bigr\}$, whose probability is $\mathbb{P}(B) = \frac1{k^3}$.
\item Bichromatic induced cycles of length four: Let $u_1-u_2-u_3-u_4-u_1$ be an induced $4$-cycle such that neither
of the opposite pairs, $\{u_1,u_3\}$ or $\{u_2,u_4\}$, is a special pair. Define the event
$C := \{f(u_1) = f(u_3), \, f(u_2) = f(u_4) \}$, whose probability is $\mathbb{P}(C) = \frac1{k^2}$.
\item For every special pair $p = \{u,v\}$ (recall that $p \notin E$), define $D = D_p := \{f(u)=f(v)\}$. Then,
$\mathbb{P}(D) = \frac1k$.
\end{enumerate}
The next lemma shows that excluding these four types of bad events guarantees that the coloring is acyclic.
\begin{lemma}
\label{lemma: 4 types of events suffice}
{\em If none of the bad events of Types~1--4 occurs, then the coloring is acyclic.}
\begin{proof}
Since no Type~1 event occurs, the vertex coloring in $G$ is proper.

Suppose, for the sake of contradiction, that there exists a bichromatic cycle, and let $C$ be a shortest such cycle.
Since the coloring is proper, $C$ has even length. Hence it suffices to consider the cases $|C|=4$ and $|C|\geq 6$.
Furthermore, $C$ is induced; otherwise, a chord of $C$ together with a subpath of $C$ would yield a shorter bichromatic
cycle, contradicting the choice of $C$. We next consider the following possibilities and rule out each of them.
\begin{enumerate}[(a)]
\item Suppose that $C$ is of length at least 6. Take five consecutive vertices
$u_1 \text{ -- } u_2 \text{ -- } u_3 \text{ -- } u_4 \text{ -- } u_5$
on the cycle. Since the cycle is bichromatic, $f(u_1)=f(u_3)=f(u_5)$ and $f(u_2)=f(u_4)$.
Since $C$ is induced, these vertices form an induced path. Thus a Type-2 event occurs,
which leads to a contradiction.
\item Suppose that $C$ is a cycle of length~4, denoted by
$u_1 \text{ -- } u_2 \text{ -- } u_3 \text{ -- } u_4 \text{ -- } u_1$.
Since $C$ is bichromatic, we have $f(u_1)=f(u_3)$ and $f(u_2)=f(u_4)$. If one opposite
pair is special, then a Type~4 event occurs. Otherwise, neither opposite pair is special
and a Type~3 event occurs, leading again to a contradiction.
\end{enumerate}
Hence no bichromatic cycle exists, so the coloring is acyclic.
\end{proof}
\end{lemma}

The next lemma appears as \cite[Lemma~2.4]{AlonMR91}, and its proof is omitted here.
\begin{lemma}[Lemma~2.4 in \cite{AlonMR91}]
\label{lemma: containedness in the 4 types of events}
{\em A vertex in $V$ belongs to at most $\Delta$ events of Type~1, $3\Delta^4$ events of Type~2,
$\Delta^{8/3}$ events of Type~3, and $\Delta^{4/3}$ events of Type~4.}
\end{lemma}

Consider a dependency graph whose vertices are the Type~1--4 bad events, where two such events are
adjacent whenever their corresponding subsets of vertices of $G$ intersect.
In light of Lemma~\ref{lemma: containedness in the 4 types of events}, consider the following asymmetric matrix:
\begin{align}
\label{eq: M matrix}
\mathbf{M}=
\begin{pmatrix}
2\Delta & 6\Delta^4 & 2\Delta^{8/3} & 2\Delta^{4/3}\\
5\Delta & 15\Delta^4 & 5\Delta^{8/3} & 5\Delta^{4/3}\\
4\Delta & 12\Delta^4 & 4\Delta^{8/3} & 4\Delta^{4/3}\\
2\Delta & 6\Delta^4 & 2\Delta^{8/3} & 2\Delta^{4/3}
\end{pmatrix},
\end{align}
whose $(i,j)$ entry, with $1 \leq i,j \leq 4$, is an upper bound on the number of Type~$j$ events that are adjacent to
a fixed Type~$i$ event in the dependency graph (see \cite[Lemma~2.5]{AlonMR91}).

For the purpose of using the LLL, the following values are assigned to the four types of events as above:
\begin{align}
\label{eq: x values for LLL}
x_A=x_D=\frac2k,  \qquad x_B=\frac2{k^3},  \qquad  x_C=\frac2{k^2}.
\end{align}
Recall that the probabilities of these bad events are given by $\mathbb{P}(A) = \frac1k = \mathbb{P}(D)$,
$\mathbb{P}(B) = \frac1{k^3}$, and $\mathbb{P}(C) = \frac1{k^2}$.
By replacing the exact dependency counts with the upper bounds given by the matrix $\mathbf{M}$ in
\eqref{eq: M matrix}, the conditions in \eqref{eq1: condition} are implied by the following three inequalities:
\begin{align}
\label{eq1: 03.06.26}
& \frac1k \leq \frac2k \, \Biggl(1-\frac2k\Biggr)^{2\Delta+2\Delta^{4/3}} \, \Biggl(1-\frac2{k^3}\Biggr)^{6\Delta^4} \,
\Biggl(1-\frac2{k^2}\Biggr)^{2\Delta^{8/3}}, \\[0.1cm]
\label{eq2: 03.06.26}
& \frac1{k^3} \leq \frac2{k^3} \, \Biggl(1-\frac2k\Biggr)^{5\Delta+5\Delta^{4/3}} \,
\Biggl(1-\frac2{k^3}\Biggr)^{15\Delta^4} \, \Biggl(1-\frac2{k^2}\Biggr)^{5\Delta^{8/3}}, \\[0.1cm]
\label{eq3: 03.06.26}
& \frac1{k^2} \leq \frac2{k^2} \, \Biggl(1-\frac2k \Biggr)^{4\Delta+4\Delta^{4/3}} \,
\Biggl(1-\frac2{k^3}\Biggr)^{12\Delta^4} \Biggl(1-\frac2{k^2}\Biggr)^{4\Delta^{8/3}}.
\end{align}
These inequalities are verified in \cite[p.~282]{AlonMR91} to hold for
$k=\lceil 50\Delta^{4/3}\rceil$. By the LLL, with positive probability,
none of the Type~1--4 bad events occurs. Hence, by Lemma~\ref{lemma: 4 types of events suffice},
there exists an acyclic coloring of $G$.
\end{proof}

\begin{remark}[Key idea in the proof of Theorem~\ref{theorem: AlonMR91}]
\label{remark: key idea - acyclic coloring}
{\em The heart of the argument is the introduction of the threshold $\Delta^{2/3}$
for the number of common neighbors for special pairs. By Lemma~\ref{lemma: special pairs}, 
the number of special pairs incident to a vertex is at most $\Delta^{4/3}$,
while every non-special pair of nonadjacent vertices has at most $\Delta^{2/3}$ 
common neighbors by definition, and therefore participates in at most $\Delta^{2/3}$ 
relevant induced $4$-cycles. The equality $\Delta^{2/3} \cdot \Delta^{4/3}=\Delta^2$ 
is exactly what makes the LLL inequalities compatible with a color set of size that 
scales like $\Delta^{4/3}$, producing the celebrated bound in \eqref{eq: AlonMR91}.}
\end{remark}

The LLL (Theorem~\ref{theorem: LLL}) has another noteworthy application in \cite{AlonMR91},
leading to the following result.
\begin{theorem}[Theorem~1.3 of \cite{AlonMR91}]
\label{theorem - AlonMR91, Theorem 1.3}
{\em Let $G$ be a finite, simple, and undirected graph with maximum degree $\Delta \geq 1$,
and suppose that for some integer $\gamma \geq 1$, the graph $G$ contains no copy of $K_{2, \gamma+1}$
in which the two vertices in the class of size~2 are nonadjacent. Then,
\begin{align}
\label{eq: AlonMR91, Theorem 1.3}
A(G) \leq \bigl\lceil \, 32 \sqrt{\gamma} \, \Delta \bigr\rceil.
\end{align}}
\end{theorem}
\smallskip
The reader is referred to \cite[Section~3]{AlonMR91} for a proof of Theorem~\ref{theorem - AlonMR91, Theorem 1.3}.
In its setting, $A(G)$ scales at most linearly with $\Delta$. Hence, if the girth
(length of shortest cycle) of $G$ is at least~5, then $A(G) = O(\Delta)$.

\smallskip
We close this subsection by presenting a tightened version of the upper bound on
the acyclic chromatic number in Theorem~\ref{theorem: AlonMR91}.
\begin{proposition}
\label{proposition: AlonMR91 tightening}
{\em For every finite, simple, and undirected graph $G$ of maximum degree $\Delta$, its acyclic chromatic number satisfies
\begin{align}
\label{eq: AlonMR91 - tightened}
A(G) \leq \lceil \, 20.571 \; \Delta (\sqrt[3]{\Delta}+1) \, \rceil.
\end{align}}
\end{proposition}
\begin{proof}
The proof of Theorem~\ref{theorem: AlonMR91} is not affected by the value of $k$, up to the derivation
of inequalities~\eqref{eq1: 03.06.26}--\eqref{eq3: 03.06.26}. We modify the last part of this proof to
obtain a smaller value of $k$ than $ \lceil 50\Delta^{4/3}\rceil $. Let
$k := \lceil C \, \Delta (\sqrt[3]{\Delta}+1) \rceil$, for some constant $C > 0$.
The validity of inequality~\eqref{eq2: 03.06.26} clearly implies that the other two
inequalities in~\eqref{eq1: 03.06.26} and \eqref{eq3: 03.06.26} hold. Inequality~\eqref{eq2: 03.06.26}
is equivalent to
\begin{align}
\label{eq4: 03.06.26}
\Biggl(1-\frac2k\Biggr)^{5\Delta+5\Delta^{4/3}} \, \Biggl(1-\frac2{k^3}\Biggr)^{15\Delta^4}
\, \Biggl(1-\frac2{k^2}\Biggr)^{5\Delta^{8/3}} \geq \frac12.
\end{align}
Referring to the left-hand side of \eqref{eq4: 03.06.26}, we have
\begin{align}
& \Biggl(1-\frac2k\Biggr)^{5\Delta+5\Delta^{4/3}} \,
\biggl(1-\frac2{k^3}\biggr)^{15\Delta^4} \, \biggl(1-\frac2{k^2}\biggr)^{5\Delta^{8/3}} \nonumber \\[0.05cm]
& \geq \Biggl(1-\frac{10(\Delta^{4/3}+\Delta)}{k}\Biggr) \,
\Biggl(1-\frac{30\Delta^4}{k^3}\Biggr) \, \Biggl(1-\frac{10\Delta^{8/3}}{k^2}\Biggr) \nonumber \\[0.05cm]
\label{eq5: 03.06.26}
& > \Biggl(1-\frac{10}{C}\Biggr) \, \Biggl(1-\frac{30}{C^3}\Biggr) \, \Biggl(1-\frac{10}{C^2}\Biggr),
\end{align}
where the first inequality follows from Bernoulli's inequality $(1-x)^n \geq 1-nx$ for $x \in [0,1]$ and $n \in \mathbb{N}$,
and the last inequality follows from the definition $k = \lceil C (\Delta^{4/3} + \Delta)$, for some $C>0$,
which implies that $k \geq C (\Delta^{4/3} + \Delta)$, $k^2 > C^2 \Delta^{8/3}$, and $k^3 > C^3 \Delta^4$.
It therefore suffices that the right-hand side of \eqref{eq5: 03.06.26} be at least $\tfrac12$
to imply that \eqref{eq4: 03.06.26} holds. The smallest $C$ for which the latter condition is satisfied
is 
\[
C_{\min} = 20.5707 \ldots \;, 
\]
which yields inequality~\eqref{eq: AlonMR91 - tightened}.
\end{proof}
The tightened upper bound in \eqref{eq: AlonMR91 - tightened} improves upon the bound in \eqref{eq: AlonMR91}, 
and the improvement increases from $16\%$ for $\Delta=1$ to a limiting value of $58.9\%$ as $\Delta \to \infty$.

\section{The Moser--Tardos algorithm and the Lov\'{a}sz Local Lemma in the variable setting}
\label{section: Moser--Tardos Algorithm}

In this section, we present a constructive version of the LLL in the variable setting,
due to Moser and Tardos~\cite{MoserTardos10}. In this framework, the bad events depend
on a family of mutually independent random variables, and dependencies among events
arise through shared variables.

While the LLL guarantees the existence of an assignment avoiding all bad events under
suitable conditions, its standard form is non-constructive. The Moser--Tardos algorithm
provides an efficient randomized procedure for finding such an assignment by iteratively
resampling the variables associated with violated events.

The analysis relies on the dependency structure among events, captured by proper witness trees,
and yields an explicit upper bound on the expected number of resampling steps in the algorithm.
This implies that the algorithm terminates almost surely and provides quantitative control over
its expected running time.

\subsection{The variable setting}
\label{subsection: Variable Setting}
Let $\{X_i\}_{i=1}^n$ be mutually independent discrete random variables, where each $X_i$
takes values in a finite set. Let $\mathcal{A} = \{A_1,\dots,A_m\}$ be a finite family of
bad events, where each event $A \in \mathcal{A}$ depends only on a subset of
$\{X_i\}_{i=1}^n$. Denote by $\mathrm{vbl}(A)$ the set of variables on which $A$ depends.
The dependency graph for $\mathcal{A}$ is the graph with vertex set $\mathcal{A}$, in which
two distinct events $A, B \in \mathcal{A}$ are adjacent if and only if
$\mathrm{vbl}(A) \cap \mathrm{vbl}(B) \neq \varnothing$. In that case, we write $B \sim A$.

\subsection{The Lov\'{a}sz Local Lemma in the variable setting}
\label{subsection: LLL - variable setting}
The following result specializes Theorem~\ref{theorem: LLL} to the variable setting, providing
a sufficient condition under which none of the events in $\mathcal{A}$ occurs with positive probability.
\begin{theorem}[Lov\'{a}sz Local Lemma (LLL): Variable Setting]
{\em Suppose there exist numbers $x(A) \in [0,1)$ for each $A \in \mathcal{A}$ such that
\begin{align}
\label{eq: LLL - variable setting}
\mathbb{P}(A) \leq x(A) \prod_{B: \, B \sim A} (1 - x(B))
\qquad \forall \, A \in \mathcal{A}.
\end{align}
Then
\begin{align}
\label{eq2: LLL - variable setting}
\mathbb{P}\left( \bigcap_{A \in \mathcal{A}} \overline{A} \right) > 0.
\end{align}}
\end{theorem}

\medskip 
The importance of the variable setting of the LLL stems from the observation that many of its 
applications are naturally formulated in terms of underlying independent random variables. Notable 
examples include CNF satisfiability (Section~\ref{subsection: k-SAT problems}), Ramsey numbers 
(Sections~\ref{subsection: diagonal Ramsey numbers} and~\ref{subsection: off-diagonal Ramsey numbers}), 
and hypergraph coloring (Section~\ref{subsection: Coloring hypergraphs}). The Moser--Tardos algorithm 
gives a constructive counterpart of the LLL in this setting.

\subsection{The Moser--Tardos algorithm}
\label{subsection: MT}
The Moser--Tardos algorithm proceeds as follows~\cite{MoserTardos10}:
\begin{itemize}
\item Sample all variables independently.
\item While there exists an event $A \in \mathcal{A}$ that occurs, select such an $A$ and resample all
the variables in $\mathrm{vbl}(A)$.
\end{itemize}
The procedure terminates when no event in $\mathcal{A}$ occurs.

\subsection{Performance of the Moser--Tardos algorithm}
\label{subsection: performance MT}
\begin{theorem}[\hspace*{-0.12cm} \cite{MoserTardos10}]
\label{theorem: Moser--Tardos}
{\em Under the condition in \eqref{eq: LLL - variable setting}, the Moser--Tardos algorithm terminates almost surely.
Moreover, for each $A \in \mathcal{A}$, the expected number of resamplings of $A$ during the execution of the algorithm
is at most $\frac{x(A)}{1-x(A)}$. Consequently, the expected total number of resamplings until termination is at most
\[
\underset{A \in \mathcal{A}}{\sum} \frac{x(A)}{1-x(A)},
\]
and the assertion in \eqref{eq2: LLL - variable setting} also holds.}
\end{theorem}

\begin{corollary}[Symmetric Moser--Tardos Criterion]
\label{corollary: symmetric-MT}
{\em Suppose that $\mathbb{P}(A) \leq p$ for all $A \in \mathcal{A}$, and that each event in $\mathcal{A}$
is adjacent to at most $d$ other events in the dependency graph. If $ep(d+1) \leq 1$, then the Moser--Tardos
algorithm terminates almost surely. Moreover, if $d \geq 1$, then the expected total number of resamplings
is at most $\frac{|\mathcal{A}|}{d}$.}
\end{corollary}

\begin{proof}
We apply Theorem~\ref{theorem: Moser--Tardos} with
\[
x(A) = \frac{1}{d+1}, \qquad A \in \mathcal{A}.
\]
As in the proof of Theorem~\ref{theorem: LLL - symmetric}, the condition $ep(d+1) \leq 1$ implies that
\[
\mathbb{P}(A) \leq x(A) \prod_{B \sim A} (1-x(B)),
\qquad \forall \, A \in \mathcal{A}.
\]
Hence, Theorem~\ref{theorem: Moser--Tardos} applies, and for each $A \in \mathcal{A}$,
\[
\mathbb{E}[N_A] \leq \frac{x(A)}{1-x(A)} = \frac{1}{d},
\]
where $N_A$ denotes the number of times $A$ is selected for resampling (equivalently, the number of
times all the variables in $\mathrm{vbl}(A)$ are resampled due to the selection of $A$) during the execution
of the algorithm. Some variables in $\mathrm{vbl}(A)$ may also be resampled when other events
are selected (due to overlaps among the variables of distinct events), and such resampling may cause $A$
to occur again. Summing over all $A \in \mathcal A$, we obtain
\[
\mathbb{E}\left[\text{total number of resamplings}\right]
= \sum_{A\in\mathcal A} \mathbb{E}[N_A]
\leq \frac{|\mathcal A|}{d}.
\]
In particular, the Moser--Tardos algorithm terminates almost surely.
\end{proof}

\subsection{Proper witness trees}
\label{subsection: proper witness trees}

We introduce the notion of a proper witness tree, which is used in the proof of Theorem~\ref{theorem: Moser--Tardos}.
We consider rooted trees whose vertices are labeled by events in $\mathcal{A}$, writing $[u] \in \mathcal{A}$ for the
label of vertex~$u$.

\begin{definition}
\label{definition: proper witness tree}
{\em A \emph{proper witness tree} is a rooted tree $T$ whose vertices are labeled by events in $\mathcal{A}$ with
the following properties:
\begin{enumerate}
\item For every edge $(u,v)$, where $v$ is a child of $u$, the labels of $u$ and $v$ share a common variable, i.e.,
\[
\mathrm{vbl}([u]) \cap \mathrm{vbl}([v]) \neq \varnothing.
\]
\item For every vertex $u$, the children of $u$ have pairwise distinct labels.
\end{enumerate}}
\end{definition}

\noindent
In Definition~\ref{definition: proper witness tree}, the labels need not be globally unique: the same event may label multiple
vertices in a proper witness tree, as distinctness is only required among the children of any vertex.

\begin{remark}
\label{remark: proper witness trees}
{\em The first condition in Definition~\ref{definition: proper witness tree} reflects the dependency structure: a child corresponds
to an earlier resampling of an event that shares a common variable with its parent, and is therefore relevant to the dependency chain
leading to the root. Resamplings of events depending on disjoint sets of variables may occur during the execution of the algorithm,
but they are included in the proper witness tree only if they are connected to the root through a chain of overlapping
events. In particular, every vertex in the tree is connected to the root via a path of overlapping events.
The second condition ensures that no event appears more than once among the children of the same vertex.}
\end{remark}

\noindent
Let $\mathcal{T}_A$ be the set of finite, proper witness trees whose root is labeled by $A \in \mathcal{A}$, i.e.,
all rooted trees that may arise from the backward construction in some execution, where each vertex is connected
to the root via a path of overlapping events.

\smallskip
\noindent
Fix a random execution of the Moser--Tardos algorithm, and let $A^{(1)},A^{(2)},A^{(3)},\dots$
be the sequence of resampled events. For each resampling step $t$, the proper witness tree $W_t$ is uniquely determined
by the prefix $A^{(1)},\dots,A^{(t)}$ via the following backward construction.

Starting from a root labeled $A^{(t)}$, the indices $t-1, t-2, \dots, 1$ are processed in decreasing order.
A vertex labeled $A^{(s)}$ is added to the tree if and only if $A^{(s)}$ shares a common variable with the label
of a vertex already present in the tree; in that case, it is attached as a child of the deepest (i.e., farthest
from the root) such vertex. If there is more than one such deepest vertex, then a fixed deterministic
tie-breaking rule is used to select one of them.

Every edge in $W_t$ connects two events that share a common variable, so every vertex in $W_t$ is connected
to the root by a path of overlapping events. Hence, $W_t$ is a proper witness tree with root labeled $A^{(t)}$.

We next relate resampling steps to proper witness trees. In counting the number of resamplings of an event, each
proper witness tree can appear at most once during a single execution; this is established in the following lemma.

\begin{lemma}[Distinct proper witness trees]
\label{lemma: distinct proper witness trees}
{\em In a fixed execution, for every pair of distinct resampling steps $s < t$, we have $W_s \neq W_t$.
In particular, each proper witness tree $\tau$ appears at most once in the execution.}
\end{lemma}

\begin{proof}
Fix two distinct resampling steps $s<t$. The proper witness tree $W_t$ is constructed by processing the
sequence $A^{(1)}, \ldots, A^{(t)}$ in reverse order, and its root is labeled $A^{(t)}$. When the
construction reaches step $s$, either $A^{(s)}$ is inserted into $W_t$, or it is skipped.
\begin{enumerate}
\item If $A^{(s)}$ is inserted into $W_t$, then the construction creates a vertex labeled $A^{(s)}$ in $W_t$.
Since $s<t$, this vertex is not the root of $W_t$. On the other hand, the root of $W_s$ is labeled
$A^{(s)}$. Therefore, $W_s$ and $W_t$ cannot coincide as rooted labeled trees, and hence $W_s \neq W_t$.

\item If $A^{(s)}$ is not inserted into $W_t$, then $A^{(s)}$ does not overlap any label already
present in the partial proper witness tree. In particular, it does not overlap the root label $A^{(t)}$.
Hence $A^{(s)} \neq A^{(t)}$, so $W_s$ and $W_t$ have different root labels, and therefore are distinct.
\end{enumerate}
Thus, in both cases $W_s\neq W_t$.
\end{proof}

By Lemma~\ref{lemma: distinct proper witness trees}, distinct resamplings of $A$
produce distinct proper witness trees in $\mathcal{T}_A$. Hence, the number of
resamplings of $A$ equals the number of proper witness trees in $\mathcal{T}_A$
that occur during the execution.

Let $N_A$ denote the number of resamplings of $A$ during this execution, i.e., the number of times $A$ is selected
for resampling. Then
\[
N_A = \sum_{\tau \in \mathcal{T}_A} \mathbf{1}_{\{\exists\, t \text{ such that } W_t = \tau\}},
\]
and therefore, by linearity of expectation,
\begin{align}
\label{eq: expected E[N_A]}
\mathbb{E}[N_A] = \sum_{\tau \in \mathcal{T}_A} \mathbb{P}\left(\exists\, t \text{ such that } W_t = \tau \right).
\end{align}

\subsection{Witness Tree Lemma}
\label{subsection: Witness Tree Lemma}

A proper witness tree $\tau$ is said to appear if there exists a resampling step $t$ such that
$W_t = \tau$. We now derive an upper bound on the probability that a fixed proper witness tree
appears during an execution of the Moser--Tardos algorithm. This bound is a key ingredient in
the analysis, as it controls the expected number of resamplings via the summation over all proper
witness trees on the right-hand side of~\eqref{eq: expected E[N_A]}.
\begin{lemma}[Witness Tree Lemma]
\label{lemma: witness tree lemma}
{\em For every proper witness tree $\tau$,
\begin{align}
\label{eq: witness tree lemma}
\mathbb{P}(\tau \text{ appears}) \leq \prod_{v \in V(\tau)} \mathbb{P}([v]),
\end{align}
where $[v]$ denotes the label of vertex $v$.}
\end{lemma}
\begin{proof}
We use the standard \emph{resampling-table} representation of the algorithm.
For each variable $X_i$, let $X_i^{(0)},X_i^{(1)},X_i^{(2)},\dots$ be an infinite
sequence of samples, each distributed as $X_i$, and assume that all these samples
are mutually independent over all $i$ and all indices.
The Moser--Tardos algorithm may be viewed as a deterministic procedure on these tables:
initially the value of $X_i$ is $X_i^{(0)}$, and whenever an event $A$ is resampled,
the algorithm advances by one step in the table of each variable in $\mathrm{vbl}(A)$.

\noindent
Fix a proper witness tree $\tau$, and let $A_v := [v]$ denote the label of vertex
$v \in V(\tau)$. We shall show that
\begin{align}
\label{eq0: 28.04.2026}
\mathbb{P}(\tau \text{ appears}) \leq \prod_{v \in V(\tau)} \mathbb{P}(A_v).
\end{align}
Consider the following checking procedure for $\tau$. Initially, for each variable $X_i$,
the current table entry is $X_i^{(0)}$. Process the vertices of $\tau$ in any order
$v_1, \dots, v_m$ such that each vertex is processed before its children.
Processing a vertex $v_r$ means that we inspect the current table entries of the
variables in $\mathrm{vbl}(A_{v_r})$ and check whether the event $A_{v_r}$ occurs under
these entries. If $A_{v_r}$ occurs, advance by one step in the table of each variable in
$\mathrm{vbl}(A_{v_r})$ and continue; otherwise, the checking procedure fails.

\noindent
Let $F_r$ be the event that the check at $v_r$ succeeds. At the moment $v_r$ is processed,
conditioned on $F_1 \cap \cdots \cap F_{r-1}$, the table entries inspected for the variables
in $\mathrm{vbl}(A_{v_r})$ are fresh, independent samples with the same joint distribution as
the variables on which $A_{v_r}$ depends. Hence,
\begin{align}
\label{eq1: 28.04.2026}
\mathbb{P}\bigl(F_r \mid F_1 \cap \cdots \cap F_{r-1}\bigr)
= \mathbb{P}(A_{v_r}),
\end{align}
where the conditioning on the left-hand side of \eqref{eq1: 28.04.2026} is omitted for $r=1$.
Therefore, by the chain rule,
\begin{align}
\mathbb{P}(\text{the checking procedure for } \tau \text{ succeeds})
&= \mathbb{P}\left(\bigcap_{r=1}^m F_r\right) \nonumber \\
&= \prod_{r=1}^m \mathbb{P}\bigl(F_r \mid F_1 \cap \cdots \cap F_{r-1}\bigr) \nonumber \\
\label{eq2: 28.04.2026}
&= \prod_{v \in V(\tau)} \mathbb{P}(A_v),
\end{align}
where the last equality holds by \eqref{eq1: 28.04.2026}.

\smallskip
\noindent
We next show that
\begin{align}
\label{eq3: 28.04.2026}
\{\tau \text{ appears}\} \subseteq \{\text{the checking procedure for } \tau \text{ succeeds}\}.
\end{align}
Suppose that $\tau$ appears during an execution of the Moser--Tardos algorithm, say
$\tau = W_t$. The backward construction of $W_t$ associates each vertex of $\tau$ with a
resampling step at which the corresponding event occurred. Moreover, when the sequence is
processed in reverse order, if two resampled events share a common variable, then the earlier
resampling is attached as a descendant of the later one.
We now run the checking procedure on $\tau$. Since the vertices are processed from the root
towards the leaves, the advances of the table indices before processing a vertex $v$ account
exactly for the later resamplings involving the same variables, represented by the ancestors of $v$.
Therefore, when $v$ is processed, the table entries inspected are precisely
those corresponding to the occurrence of $A_v$ at its associated resampling step in the
execution. Since the algorithm resamples an event only when it occurs, each such check
succeeds. Hence, if $\tau$ appears, then the checking procedure for $\tau$ succeeds,
which proves \eqref{eq3: 28.04.2026}.

\noindent
Finally, combining \eqref{eq2: 28.04.2026} and \eqref{eq3: 28.04.2026}, it follows that
\[
\mathbb{P}(\tau \text{ appears})
\leq \mathbb{P}(\text{the checking procedure for } \tau \text{ succeeds})
= \prod_{v \in V(\tau)} \mathbb{P}(A_v),
\]
which proves inequality~\eqref{eq0: 28.04.2026}.
\end{proof}

\subsection{Bounding the expected number of resamplings}
\label{subsection: Expected Number of Resamplings}

\noindent
Combining \eqref{eq: expected E[N_A]} and \eqref{eq: witness tree lemma}
implies that, for all $A \in \mathcal{A}$,
\begin{align}
\label{eq: UB E[N_A]}
\mathbb{E}[N_A] \leq \sum_{\tau \in \mathcal{T}_A} \prod_{v \in V(\tau)} \mathbb{P}([v]).
\end{align}
We next bound this sum using a branching process.

\subsection{Branching process}
\label{subsection: Branching Process}

A Galton–Watson branching process is a stochastic process in which each vertex independently generates 
a random number of children according to a common offspring distribution, thereby generating a random 
rooted tree.

Here and throughout, $B \sim C$ denotes that the events $B$ and $C$ depend on a
common variable, i.e., $\mathrm{vbl}(B) \cap \mathrm{vbl}(C) \neq \varnothing$.
For an event $B \in \mathcal{A}$, let
\[
\Gamma(B) := \{ C \in \mathcal{A} \setminus \{B\} : C \sim B \}
\]
be the set of events, excluding $B$ itself, that share at least one variable with $B$;
equivalently, $\Gamma(B)$ is the set of neighbors of $B$ in the dependency graph.
Furthermore, let $\Gamma^{+}(B) := \Gamma(B) \cup \{B\}$.

\smallskip
Fix $A \in \mathcal{A}$. We define a branching process that generates proper witness trees
rooted at $A$.
\begin{itemize}
\item Start with a root labeled $A$.
\item For each node labeled $B$ and each event $C \in \Gamma^{+}(B)$, independently include a child
labeled $C$ with probability $x(C)$.
\end{itemize}
The process continues until it becomes extinct, i.e., until no new vertices are generated in some generation
(depending on the probabilities, this may occur with probability strictly less than one).
This randomized process generates labeled rooted trees in which every edge connects overlapping events
and children have distinct labels. In particular, every tree produced by this process is
a proper witness tree (though not every such tree necessarily arises from an execution of the
algorithm). In this randomized process, each node labeled $B$ generates a random subset
$S \subseteq \Gamma^{+}(B)$ with probability
\[
\prod_{C \in S} x(C) \prod_{C \in \Gamma^{+}(B)\setminus S} (1 - x(C)).
\]

\noindent
We now compute the probability that this process generates a given proper witness tree.
\begin{lemma}
\label{lemma: branching-process}
{\em Let $\tau$ be a proper witness tree rooted at $A$. Then, the probability that the branching
process produces $\tau$ is
\begin{align}
\label{eq: branching-process}
p_\tau = \frac{1 - x(A)}{x(A)}
\prod_{v \in V(\tau)} \left( x([v]) \prod_{C \in \Gamma([v])} (1 - x(C)) \right).
\end{align}}
\end{lemma}
\begin{proof}
Fix a proper witness tree $\tau$ rooted at $A$. We compute the probability that the branching
process produces exactly $\tau$.

In the branching process, each vertex independently selects its set of children. Therefore, the
probability of generating the tree $\tau$ is equal to the product, over all vertices $v \in V(\tau)$,
of the probabilities that each vertex selects precisely its children as prescribed by $\tau$. For a
vertex $v \in V(\tau)$, let $\mathrm{children}(v)$ denote the set of labels of the children of $v$
in the tree $\tau$.

Let $v \in V(\tau)$ be a vertex with label $[v]$. The set of its children in $\tau$ is a subset
of $\Gamma^{+}([v])$, i.e., $\mathrm{children}(v) \subseteq \Gamma^{+}([v])$.
By the construction of the branching process, each vertex $C \in \Gamma^{+}([v])$ is included as a child
independently with probability $x(C)$. Hence, the probability that $v$ selects exactly its children
in $\tau$ is
\[
\prod_{C \in \, \mathrm{children}(v)} x(C)
\prod_{C \in \, \Gamma^{+}([v]) \setminus \, \mathrm{children}(v)} (1 - x(C)).
\]
Multiplying these probabilities over all vertices $v \in V(\tau)$, we obtain
\[
p_\tau = \prod_{v \in V(\tau)} \left(
\prod_{C \in \, \mathrm{children}(v)} x(C)
\prod_{C \in \, \Gamma^{+}([v]) \setminus \, \mathrm{children}(v)} (1 - x(C)) \right).
\]
For each vertex $v$, we rewrite the second factor as
\[
\prod_{C \in \, \Gamma^{+}([v]) \setminus \, \mathrm{children}(v)} (1 - x(C))
= \frac{\underset{C \in \, \Gamma^{+}([v])}{\prod} (1 - x(C))}
{\underset{C \in \, \mathrm{children}(v)}{\prod} (1 - x(C))}.
\]
Substituting this identity gives
\[
p_\tau = \prod_{v \in V(\tau)} \left(
\prod_{C \in \, \Gamma^{+}([v])} (1 - x(C)) \right)
\cdot \prod_{v \in V(\tau)} \; \prod_{C \in \, \mathrm{children}(v)} \frac{x(C)}{1 - x(C)}.
\]
Let $r$ denote the root of $\tau$. Then $[r]=A$.
Since each non-root vertex appears exactly once as a child of its parent,
\[
\prod_{v \in V(\tau)} \prod_{C \in \, \mathrm{children}(v)} \frac{x(C)}{1 - x(C)}
= \prod_{v \in V(\tau) \setminus \{r\}} \frac{x([v])}{1 - x([v])}.
\]
Rewriting this as a product over all vertices gives 
\[
\prod_{v \in V(\tau) \setminus \{r\}} \frac{x([v])}{1 - x([v])}
= \frac{1-x(A)}{x(A)} \prod_{v \in V(\tau)} \frac{x([v])}{1 - x([v])},
\]
and therefore
\begin{align*}
p_\tau = \frac{1-x(A)}{x(A)}
\prod_{v \in V(\tau)} \left( \frac{x([v])}{1-x([v])}
\prod_{C \in \Gamma^{+}([v])} (1-x(C)) \right).
\end{align*}
Finally, since $\Gamma^{+}([v]) = \Gamma([v]) \cup \{[v]\}$ and the union is disjoint, equality \eqref{eq: branching-process} follows.
\end{proof}

\smallskip
\noindent
Since the Galton-Watson branching process produces at most one tree (not necessarily from $\mathcal{T}_{A}$
since it can produce an infinite tree), we have
\[
\sum_{\tau \in \mathcal{T}_{A}} p_\tau \leq 1.
\]
Therefore, by Lemma~\ref{lemma: branching-process} and since $x([v]) \in [0,1)$,
\begin{align}
\label{eq: branching-process-bound}
\sum_{\tau \in \mathcal{T}_{A}}
\prod_{v \in V(\tau)} \left( x([v]) \prod_{C \in \Gamma([v])} (1 - x(C)) \right)
\leq \frac{x(A)}{1 - x(A)}.
\end{align}

\subsection{Completion of the proof of Theorem~\ref{theorem: Moser--Tardos}}
\label{subsection: completion of MT theorem}

\noindent
By condition~\eqref{eq: LLL - variable setting}, for every vertex $v \in V(\tau)$,
\[
\mathbb{P}([v]) \leq x([v]) \prod_{C \in \Gamma([v])} (1 - x(C)).
\]
Hence, for every $\tau \in \mathcal{T}_{A}$ with $A \in \mathcal{A}$,
\begin{align}
\label{eq: UB on product of prob.}
\prod_{v \in V(\tau)} \mathbb{P}([v])
\leq \prod_{v \in V(\tau)} \left( x([v]) \prod_{C \sim [v]} (1 - x(C)) \right).
\end{align}
Combining \eqref{eq: UB E[N_A]}, \eqref{eq: branching-process-bound}, and \eqref{eq: UB on product of prob.}, we obtain
\[
\mathbb{E}[N_{A}]
\leq \sum_{\tau \in \mathcal{T}_{A}} \prod_{v \in V(\tau)} \mathbb{P}([v])
\leq \frac{x(A)}{1 - x(A)},
\]
and summing over all events $A \in \mathcal{A}$ implies that
\[
\mathbb{E}[\text{total resamplings}]
\leq \sum_{A \in \mathcal{A}} \frac{x(A)}{1 - x(A)} < \infty.
\]
Consequently, since the expected total number of resamplings is finite, the algorithm terminates
almost surely after finitely many resampling steps and produces an assignment in which none of
the events in $\mathcal{A}$ occurs. In particular, $\mathbb{P}\left( \bigcap_{A \in \mathcal{A}} \overline{A} \right) > 0.$

\subsection{Shearer's Bound and Beyond for the Moser--Tardos Algorithm}
\label{subsection: Shearer's Bound and Beyond for the Moser--Tardos Algorithm}

We close this section by noting that, in the variable setting considered here, recent works
\cite{KolipakaS11,HeLS23,HeLLWX26} have shown that the efficient region of the Moser--Tardos
algorithm reaches, and in some cases exceeds, Shearer's bound (Theorem~\ref{theorem: Shearer}).
This does not contradict the optimality of Shearer's criterion in the abstract LLL
(Theorem~\ref{theorem: LLL}), where only the dependency graph and the event probabilities are
taken into account. In the variable model, however, one retains additional structural information
beyond the dependency graph. Specifically, the variable model is described by a bipartite
variable--event graph whose vertices are the variables and bad events, with an edge joining a
variable to a bad event whenever the latter depends on the former. The corresponding dependency
graph is then obtained by projecting this bipartite graph onto the event vertices, linking two
events whenever they depend on a common variable.

The connection between Shearer's criterion and the Moser--Tardos algorithm was established by
Kolipaka and Szegedy \cite{KolipakaS11}, who showed that, when only the dependency graph and
event probabilities are taken into account, the efficient region of the algorithm coincides with
the Shearer region. 

Building on the additional structure available in the variable setting, \cite{HeLS23} develops 
a refined analysis of the Moser--Tardos algorithm and shows that the efficient region extends 
beyond the Shearer region if the dependency graph is non-chordal, that is, if it contains an 
induced cycle of length at least four. By contrast, for chordal dependency graphs (i.e., graphs 
in which every cycle of length at least four has a chord), Shearer's bound still exactly characterizes 
the efficient region.

Subsequently, \cite{HeLLWX26} clarified the role of non-chordality by studying the relationship
between the abstract-LLL and variable-LLL boundaries. In particular, it showed that the two
boundaries coincide for trees, whereas the presence of an induced cycle of length at least four
gives rise to a genuine gap between them. Thus, while Shearer's criterion remains the optimal
threshold in the abstract dependency-graph setting, the variable model retains structural
information that is discarded when passing to the dependency graph and that can be exploited
algorithmically to obtain convergence guarantees beyond the Shearer threshold.

\section{The entropy-compression principle}
\label{section: entropy-compression principle}

The entropy-compression principle, commonly attributed to Moser \cite{Moser09} and whose name was 
coined by Tao \cite{Tao2009}, is closely related to the Moser--Tardos resampling framework for the LLL. 
The basic idea is to analyze a randomized correction algorithm through suitably defined execution logs 
that record the bad events encountered and the corresponding resampling steps. Suppose that every 
execution surviving at least $t$ steps can be associated with a pair consisting of an execution log 
and the state of the algorithm after step $t$, and that this correspondence is injective. If, moreover, 
the total number of such log--state pairs grows at a strictly smaller exponential rate than the number 
of possible random input sequences, then the proportion of inputs generating long executions decays 
exponentially with $t$. Consequently, the algorithm terminates almost surely.

This viewpoint has proved particularly fruitful in probabilistic combinatorics \cite{EsperetP13,DujmovicJKW16}, 
where it often yields constructive existence proofs and quantitative bounds through direct counting arguments 
on the possible execution histories of randomized algorithms.

The following theorem abstracts the counting argument underlying the entropy-compression method introduced by 
Moser \cite{Moser09}; see also Tao \cite{Tao2009}.
\begin{theorem}[Entropy-compression principle]
\label{theorem: entropy-compression principle}
{\em Let  $\Omega$ be a finite set of size $q$, and consider a randomized algorithm whose 
random choices are independent and uniformly distributed over $\Omega$. Assume that once 
the sequence of random choices is fixed, the execution of the algorithm is uniquely determined. 
For each integer $t \geq 1$, let $B_t$ be defined as 
\begin{align}
\label{eq: B-23.05.26e}
B_t:= \Bigl\{ R \in \Omega^t: \text{the execution determined by }R \text{ reaches step }t \Bigr\},
\end{align}
which is a deterministic subset of $\Omega^t$. For every $R \in B_t$, let $L_t(R) \in \mathcal{L}_t$
be a record (or log) of the information collected during the first $t$ steps of the execution, 
and let $S_t(R) \in \mathcal{S}_t$ be the state after step $t$. Assume that for every $t \geq 1$, 
the mapping 
\begin{align}
\label{eq: Phi-23.05.26d}
\Phi_t \colon B_t \to \mathcal{L}_t \times \mathcal{S}_t, \qquad \Phi_t(R)=(L_t(R),S_t(R)),
\end{align}
is injective, and that 
\begin{align}
\label{eq: 23.05.26f}
|\mathcal L_t \times \mathcal S_t| \leq C \alpha^t,
\end{align}
for some constants $C>0$ and $\alpha < q$, independent of $t$.
Then,
\begin{align}
\label{eq: 23.05.26g}
\mathbb{P}(\text{the algorithm reaches step } t) \leq C \left( \frac{\alpha}{q} \right)^t.
\end{align}
Consequently, the algorithm terminates almost surely.}
\end{theorem}
\begin{proof}
Since, by assumption, the mapping $\Phi_t \colon B_t \to \mathcal{L}_t \times \mathcal{S}_t$ is injective,
it follows that 
\begin{align}
|B_t| \leq |\mathcal{L}_t \times \mathcal{S}_t| \leq C \alpha^t.
\end{align}
Therefore, for every integer $t \geq 1$, 
\begin{align}
\mathbb{P}(\text{the algorithm lasts at least }t\text{ steps})
& =\frac{|B_t|}{q^t} \nonumber \\
\label{eq1: 15.06.26}
& \leq C \, \biggl(\frac{\alpha}{q} \biggr)^t.
\end{align}
Since $\alpha<q$, the right-hand side of \eqref{eq1: 15.06.26} tends to~0 as $t \to \infty$. 
Let $T$ denote the execution time of the algorithm. Then, 
\begin{align}
\lim_{t \to \infty} \mathbb{P}(T > t) = 0.
\end{align}
Hence $\mathbb{P}(T < \infty) = 1$, and therefore the algorithm terminates almost surely.
\end{proof}

\begin{remark}[Connection to the Shannon entropy]
\label{remark: connection to the Shannon entropy}
{\em The terminology \emph{entropy compression} stems from the following information-theoretic 
interpretation. Since the entries of the random vector $R=(r_1,\ldots,r_t)$ are independent 
and uniformly distributed over an alphabet $\Omega$ of size $q$, the Shannon entropy
of $R$ equals
\begin{align}
\label{eq1: entropy}
\Entr(R) = t \log q.
\end{align}
Furthermore, since the map $R \mapsto (L_t,S_t)$ is injective, the pair $(L_t,S_t)$
uniquely determines the random input vector $R$. This implies that the conditional 
entropy of $R$ given $(L_t, S_t)$ satisfies 
\begin{align}
\Entr(R \hspace{-0.05cm} \mid \hspace{-0.05cm} L_t, S_t) = 0.
\end{align}
Consequently, by the chain rule for the entropy (see \cite[Theorem~2.5.1]{CoverThomas2006}), it follows that
\begin{align}
\Entr(R) & \leq \Entr(R, L_t, S_t) \nonumber \\
& = \Entr(R \hspace{-0.05cm} \mid \hspace{-0.05cm} L_t, S_t) + \Entr(L_t, S_t) \nonumber \\
\label{eq2: entropy}
&= \Entr(L_t,S_t).
\end{align}
Since the Shannon entropy of a random vector is upper bounded by the logarithm
of the number of its possible outcomes (see \cite[Theorem~2.6.4]{CoverThomas2006}),
it follows from \eqref{eq: 23.05.26f} that
\begin{align}
\Entr(L_t,S_t) & \leq \log(C \alpha^t) \nonumber \\
\label{eq3: entropy}
&= \log C + t \log \alpha.
\end{align}
Combining \eqref{eq1: entropy}--\eqref{eq3: entropy} therefore gives
\begin{align}
\label{eq: 23.05.26m}
t \log q = \Entr(R) \leq \Entr(L_t,S_t) \leq \log C+t\log\alpha.
\end{align}
Since by assumption $\alpha$ and $C$ are fixed positive constants with $\alpha < q$, inequality \eqref{eq: 23.05.26m} 
is violated for sufficiently large~$t$.
Thus, excessively long executions would imply that the random input vector $R$ can be represented by the
pair $(L_t,S_t)$, whose number of possible outcomes grows asymptotically like $\alpha^t$ with $\alpha<q$.
In this sense, the information contained in the random source is effectively compressed, contradicting the entropy
lower bound $\Entr(R)= t \log q$.}
\end{remark}

\begin{remark}[Connection to the LLL]
\label{remark: connection of entropy-compression principle to LLL}
{\em The entropy-compression principle is closely related to the Moser--Tardos algorithmic proof of the LLL. 
The entropy-compression framework applies naturally to the Moser--Tardos resampling algorithm. Indeed, 
the successive resampling decisions may be viewed as a sequence of independent random choices drawn from 
a finite alphabet, and the resulting execution history is a deterministic function of this random sequence. 
Thus, the Moser--Tardos algorithm is an instance of the abstract randomized process considered in the 
entropy-compression principle. In both settings, one considers a randomized correction procedure in which 
variables are sampled independently and local resampling steps are performed whenever a bad event occurs.

\noindent
The two approaches differ primarily in how the execution process is analyzed. The Moser--Tardos framework 
studies the resampling algorithm through witness trees and branching-process estimates, whereas entropy-compression 
arguments analyze execution histories via combinatorial logs and injective reconstruction maps.
More precisely, the entropy-compression approach shows that excessively long executions would yield an injective 
encoding of the random input into a family of descriptions whose cardinality is too small to accommodate the entropy 
of the underlying random source. Consequently, the number of possible long execution histories is exponentially 
smaller than the number of random inputs that would have to be encoded by them.}
\end{remark}

\section{The lopsided Lov\'{a}sz Local Lemma}
\label{section: Lopsided LLL}

The LLL has an important refinement in which the standard notion of a dependency 
graph is replaced by the weaker notion of a \emph{lopsidependency graph}. This 
refinement is especially useful in problems involving random permutations, matchings, 
and transversals, where the relevant bad events are generally not independent but 
satisfy a suitable negative-dependence condition. 
In Section~\ref{subsection: presentation of the LLLL} we present the lopsided LLL, 
and in Section~\ref{subsection: Latin transversals via the lopsided LLL} we apply 
it to the problem of Latin transversals.

\subsection{Formulation of the lopsided Lov\'{a}sz Local Lemma}
\label{subsection: presentation of the LLLL}

\noindent
In this subsection, we present the lopsided LLL and discuss several
related remarks. 

\begin{definition}[Lopsidependency graph]
\label{definition: lopsidependency graph}
{\em Let $\{A_i\}_{i \in I}$ be a finite family of events in a probability space.
A graph $G$ on the vertex set $I$ is called a \emph{lopsidependency graph}
for $\{A_i\}_{i \in I}$ if, for every $i \in I$ and every set
\begin{align}
\label{eq1: LLLL}
S \subseteq I \setminus \bigl(\Gamma(i)\cup\{i\}\bigr),
\end{align}
where $\Gamma(i)$ denotes the set of neighbors of $i$ in $G$, we have
\begin{align}
\label{eq2: LLLL}
\mathbb{P}\left(A_i \; \middle| \; \bigcap_{j\in S}\overline{A_j} \right) \leq \mathbb{P}(A_i).
\end{align}}
\end{definition}

\begin{theorem}[Lopsided Lov\'{a}sz Local Lemma]
\label{theorem: Lopsided LLL}
{\em Let $\{A_i\}_{i \in I}$ be a finite family of events in a probability space,
and let $G$ be a lopsidependency graph for $\{A_i\}_{i\in I}$.
If there exist numbers $\{x_i\}_{i\in I} \subseteq [0,1)$ such that
\begin{align}
\label{eq3: LLLL}
\mathbb{P}(A_i) \leq x_i \prod_{j \in \Gamma(i)}(1-x_j), \qquad \forall \, i \in I,
\end{align}
then
\begin{align}
\label{eq4: LLLL}
\mathbb{P}\left(\bigcap_{i \in I} \overline{A_i} \right) \geq \prod_{i \in I} (1-x_i) > 0.
\end{align}}
\end{theorem}
\begin{proof}
This result follows from the proof of the LLL given in Section~\ref{section: LLL and modified proof}.
Specifically, Lemma~\ref{lemma: LLL-unconditional} remains valid under the weaker assumptions 
\eqref{eq2: LLLL} and \eqref{eq3: LLLL}, with the obvious notational modification that the index set $[n]$ 
is replaced by the finite set $I$. Consequently, Step~2 of that proof remains valid under the same assumptions.
\end{proof}

\noindent
A convenient symmetric form follows from Theorem~\ref{theorem: Lopsided LLL}; compared with Theorem~\ref{theorem: LLL - symmetric}, 
it requires a weaker hypothesis.
\begin{corollary}[Symmetric lopsided Lov\'asz local lemma]
\label{corollary: symmetric Lopsided LLL}
{\em Under the assumptions of Theorem~\ref{theorem: Lopsided LLL} with $\mathbb{P}(A_i) \leq p$ for all $i \in I$, 
suppose that every vertex of $G$ has degree at most $d \geq 1$, and $ep(d+1)\leq 1$. Then
\begin{align}
\label{eq3: symmetric LLLL}
\mathbb{P}\left(\bigcap_{i \in I} \overline{A_i} \right) & \geq \left( \frac{d}{1+d} \right)^{|I|}.
\end{align}}
\end{corollary}
\begin{proof}
Apply Theorem~\ref{theorem: Lopsided LLL} with $x_i = \frac1{1+d}$ for all $i \in I$. The rest of the proof 
proceeds as in the proof of Theorem~\ref{theorem: LLL - symmetric}.
\end{proof}

\begin{remark}
\label{remark: Lopsided LLL refines LLL}
{\em The lopsided LLL is a refinement of the LLL (Theorem~\ref{theorem: LLL})
when the latter is formulated in terms of undirected dependency graphs. 
Indeed, in this formulation, each event $A_i$ is independent of the 
$\sigma$-algebra generated by the events $A_j$ with 
$j \notin \Gamma(i) \cup \{i\}$. This implies, in particular, that
\begin{align}
\label{eq5: LLLL}
\mathbb{P}\left(A_i \, \middle|\, \bigcap_{j\in S}\overline{A_j} \right) = \mathbb P(A_i),
\end{align}
for all $S\subseteq I\setminus(\Gamma(i)\cup\{i\})$, whenever the conditional
probability is well defined. Hence, every dependency graph is also a
lopsidependency graph.

The converse is not true in general. The lopsided condition requires only that 
conditioning on the non-occurrence of non-neighboring bad events does not increase 
the probability of a given bad event. Consequently, the lopsided LLL can be applicable 
even when the ordinary LLL is not. As illustrated in 
Section~\ref{subsection: Latin transversals via the lopsided LLL}, this is particularly 
useful in permutation and matching problems, where the dependence structure often prevents 
a direct application of the ordinary LLL, while still satisfying the negative-dependence 
condition \eqref{eq2: LLLL} required by the lopsided LLL.}
\end{remark}

\begin{remark}
{\em The conditions in Definition~\ref{definition: lopsidependency graph}
and Theorem~\ref{theorem: Lopsided LLL} may be formulated slightly more 
generally using auxiliary numbers $\{p_i\}_{i \in I} \subseteq [0,1)$, 
satisfying
\begin{align}
\label{eq1: 16.06.26}
p_i \geq \mathbb{P}(A_i).
\end{align}
Indeed, it suffices to assume that for every $i\in I$ and every subset 
$S$ satisfying \eqref{eq1: LLLL}, we have
\begin{align}
\label{eq2: 16.06.26}
\mathbb{P}\left(A_i \,\middle|\, \bigcap_{j\in S}\overline{A_j}\right)
\leq p_i,
\end{align}
and that there exist numbers $x_i\in[0,1)$ such that
\begin{align}
\label{eq3: 16.06.26}
p_i \leq x_i \prod_{j \in \Gamma(i)}(1-x_j),
\qquad \forall\, i \in I.
\end{align}
Note that taking $S = \varnothing$ in \eqref{eq2: 16.06.26} yields 
\eqref{eq1: 16.06.26} for every $i \in I$. The standard proof of 
the lopsided LLL then applies verbatim, yielding \eqref{eq4: LLLL}.}
\end{remark}

\subsection{Latin transversals via the lopsided Lov\'{a}sz Local Lemma}
\label{subsection: Latin transversals via the lopsided LLL}

Latin transversals are important combinatorial objects that arise naturally in
matching and assignment problems, combinatorial design theory, graph coloring,
and scheduling. Their study is closely connected to Latin squares,
permutation problems, and the probabilistic method. In particular, deriving
sufficient conditions for the existence of a Latin transversal in a given
matrix is a classical application of the lopsided LLL \cite{ErdosSpencer91}.

\begin{definition}[Latin transversal of a matrix]
\label{definition: Latin transversal}
{\em A \emph{Latin transversal} of an $n \times n$ matrix
$\mathbf A=(a_{i,j})$ is a selection of $n$ entries, one from each row
and each column, whose values are pairwise distinct. Equivalently,
it is a set of positions
\[
\{(i,\pi(i)): 1 \leq i \leq n\},
\]
where $\pi$ is a permutation of $[n]$, such that the sequence
$\{a_{i,\pi(i)}\}_{i=1}^n$ has pairwise distinct terms.}
\end{definition}

\begin{theorem}[Existence of Latin transversals]
\label{theorem: Latin transversals}
{\em Let $\mathbf{A}=(a_{i,j})$ be an $n \times n$ matrix in which every symbol
occurs at most $k$ times. If
\begin{align}
\label{eq: Latin transversals}
k\leq \frac{n-1}{4e}+1-\frac1{4n},
\end{align}
then $\mathbf{A}$ contains a Latin transversal.}
\end{theorem}

\begin{proof}
Let $\pi$ be a uniformly random permutation of $[n]$.
Denote by $\mathcal T$ the set of all ordered quadruples
$(i,j,i',j')$ satisfying $i<i'$, $j\neq j'$, and $a_{i,j}=a_{i',j'}$.
Thus, $(i,j)$ and $(i',j')$ are positions containing the same symbol 
and lying in distinct rows and distinct columns.
For every $(i,j,i',j') \in \mathcal{T}$, define the event
\begin{align}
\label{eq: event A - Latin transversals}
A_{i,j,i',j'} = \bigl\{(\pi(i),\pi(i'))=(j,j')\bigr\}.
\end{align}
We first observe that
\begin{align}
\label{eq: equivalence - Latin transversal}
\mathbf A \text{ contains a Latin transversal}
\iff
\mathbb P\left(
\bigcap_{(i,j,i',j')\in\mathcal T}
\overline{A_{i,j,i',j'}}
\right)>0.
\end{align}
We prove each implication separately.
\begin{itemize}
\item Suppose that $\mathbf{A}$ contains a Latin transversal. Then, there exists a permutation
$\pi_0$ of $[n]$ such that $a_{i,\pi_0(i)} \neq a_{i',\pi_0(i')}$ for all $i<i'$.
Hence, for every $(i,j,i',j') \in \mathcal T$, we have
$\bigl( \pi_0(i), \pi_0(i') \bigr) \neq (j, j')$ since otherwise 
\[
a_{i,\pi_0(i)}=a_{i,j}=a_{i',j'}=a_{i',\pi_0(i')}, 
\]
contradicting the fact that the selected entries are distinct. Hence, none of the events 
$A_{i,j,i',j'}$ with $(i,j,i',j') \in \mathcal{T}$ occurs when $\pi=\pi_0$.
Since $\pi$ is uniformly distributed over the $n!$ permutations
of $[n]$, we have $\mathbb{P}(\pi=\pi_0) = \frac{1}{n!}>0$. Consequently,
\begin{align}
\label{eq1: 04.06.26}
\mathbb P\left(\bigcap_{(i,j,i',j') \in \mathcal{T}} \overline{A_{i,j,i',j'}}\right)>0.
\end{align}
\item Conversely, suppose that \eqref{eq1: 04.06.26} holds.
Then there exists a permutation $\pi_0$ of $[n]$ for which no event
$A_{i,j,i',j'}$ with $(i,j,i',j') \in \mathcal{T}$ occurs. We claim that
\[
T_{\pi_0}:= \bigl\{(i,\pi_0(i)): \, i\in[n] \bigr\}
\]
is a Latin transversal. Otherwise, there exist indices $i<i'$ such that
$a_{i,\pi_0(i)} = a_{i',\pi_0(i')}$. Set $j:=\pi_0(i)$ and $j':=\pi_0(i')$.
Since $\pi_0$ is a permutation and $i \neq i'$, we have $j \neq j'$. Then,
$(i,j,i',j') \in \mathcal{T}$ and $\bigl( \pi_0(i), \pi_0(i') \bigr)=(j,j')$.
Therefore, the event $A_{i,j,i',j'}$ occurs when $\pi = \pi_0$, contradicting 
the choice of $\pi_0$. Hence the entries $\{a_{i,\pi_0(i)}\}_{i=1}^n$ are 
pairwise distinct, and therefore $T_{\pi_0}$ is a Latin transversal.
\end{itemize}

Now fix $(i,j,i',j')\in\mathcal T$. Since $\pi$ is a uniformly random
permutation of $[n]$, and $i \neq i'$, $j \neq j'$,
\[
\mathbb{P}(A_{i,j,i',j'}) = \mathbb{P}(\pi(i)=j,\, \pi(i')=j')
= \frac{(n-2)!}{n!} = \frac{1}{n(n-1)} =:p.
\]
Define a graph $G$ on the vertex set $\mathcal T$ by joining two distinct
vertices $(i,j,i',j')$ and $(\ell,m,\ell',m')$ whenever
\begin{align}
\label{eq: adjacency condition Latin transversal}
\{i,i'\}\cap\{\ell,\ell'\}\neq\varnothing
\quad\text{or}\quad
\{j,j'\}\cap\{m,m'\}\neq\varnothing.
\end{align}
By a result of Erd\H{o}s and Spencer \cite[Section~2]{ErdosSpencer91}, the graph $G$ is a
lopsidependency graph for the family of events
\[
\{A_{i,j,i',j'}:(i,j,i',j') \in \mathcal T\}.
\]
By Definition~\ref{definition: lopsidependency graph}, let $(i,j,i',j') \in \mathcal{T}$, and 
let $S \subseteq \mathcal{T}$ be a set of vertices of $G$ that are non-adjacent to $(i,j,i',j')$ in $G$. Then
\[
\mathbb P\left( A_{i,j,i',j'} \,\middle|\,
\bigcap_{(\ell,m,\ell',m')\in S}
\overline{A_{\ell,m,\ell',m'}} \right) \leq \mathbb P(A_{i,j,i',j'}),
\]
which is the lopsidependency condition required by the lopsided LLL.

It remains to derive an upper bound on the maximum degree of $G$.
Fix a vertex $(i,j,i',j')\in\mathcal T$. An adjacent vertex must contain
a position $(s,t)$ such that $s\in\{i,i'\}$ or $t\in\{j,j'\}$.
There are at most $4n$ such positions $(s,t)$.
For each such position $(s,t)$, since every symbol occurs at most $k$ times
in the matrix, there are at most $k-1$ other positions $(s',t')$ satisfying
$a_{s,t}=a_{s',t'}$.
Each admissible pair of positions $(s,t)$ and $(s',t')$ lying in distinct rows
and distinct columns determines at most one vertex of $\mathcal T$. Consequently, every
vertex of $G$ has degree at most $4n(k-1)$, i.e.,
\begin{align}
d \leq 4n(k-1).
\end{align}
By the symmetric version of the lopsided LLL (Corollary~\ref{corollary: symmetric Lopsided LLL}),
it is enough that $ep(d+1) \leq 1$ to satisfy \eqref{eq1: 04.06.26}.
Using the bounds above, this is implied by
\[
e\cdot \frac{1}{n(n-1)}\cdot \bigl(4n(k-1)+1 \bigr) \leq 1,
\]
which is equivalent to \eqref{eq: Latin transversals}.
Finally, by \eqref{eq: equivalence - Latin transversal} and \eqref{eq1: 04.06.26}, the matrix 
$\mathbf{A}$ contains a Latin transversal.
\end{proof}

\begin{remark}
\label{remark: why ordinary LLL does not apply}
{\em The symmetric version of the LLL (Theorem~\ref{theorem: LLL - symmetric}) cannot be applied 
directly in the proof of Theorem~\ref{theorem: Latin transversals}. The reason is that, under
the uniform distribution on permutations, events involving disjoint rows and
columns are generally not independent.
For instance, suppose that $i,i',\ell,\ell'$ are distinct and that
$j,j',m,m'$ are distinct. Then, for $n \geq 4$,
\[
\mathbb P\bigl(
A_{i,j,i',j'} \cap A_{\ell,m,\ell',m'}\bigr) = \frac{1}{n(n-1)(n-2)(n-3)},
\]
whereas
\[
\mathbb P(A_{i,j,i',j'}) \, \mathbb P(A_{\ell,m,\ell',m'}) = \frac{1}{n^2(n-1)^2}.
\]
Thus, such events are not independent,
even though their corresponding positions lie in disjoint rows and columns. In fact, the
former quantity is larger than the latter, showing that these events are positively correlated.
Consequently, the occurrence of one such event makes the occurrence of the other more likely,
whereas the non-occurrence of one event tends to make the occurrence of the other less likely.
The lopsided LLL is therefore essential here. Rather than requiring independence, it suffices
that conditioning on the non-occurrence of non-neighboring bad events does not increase the
probability of a given bad event. Erd\H{o}s and Spencer showed that this property holds for
the events arising from random permutations \cite{ErdosSpencer91}, which allows the lopsided 
LLL to be applied in the proof of Theorem~\ref{theorem: Latin transversals}.}
\end{remark}

\begin{remark}
\label{remark: Latin squares}
{\em Theorem~\ref{theorem: Latin transversals} does not guarantee the existence of a
Latin transversal in an $n \times n$ Latin square with $n\geq 2$, since in that
case every symbol appears exactly $n$ times. Thus $k=n$, and the inequality
\eqref{eq: Latin transversals} is violated. Latin squares of even order need not
have a Latin transversal; for example, the Latin square
$\mathbf{A}=(a_{i,j})_{i,j=1}^{\,n}$, where $a_{i,j}:=(i+j)\pmod n$, has no
Latin transversal for even $n\geq 2$. In contrast, Ryser's conjecture states
that every Latin square of odd order has a Latin transversal; see, for example,
\cite[Conjecture~3.2]{Wanless2011}.}
\end{remark}

\section{The Cluster-Expansion Lemma}
\label{section: cluster-expansion lemma}

This section presents the Cluster-Expansion Lemma, due to Bissacot, Fern\'{a}ndez,
Procacci, and Scoppola~\cite{BissacotFPS2011}. Their original proof relied on
cluster-expansion techniques from statistical physics. Subsequently, 
Pegden~\cite{Pegden14} established an algorithmic counterpart within the Moser--Tardos 
framework, and later Harvey and Vondr\'{a}k~\cite{HarveyV20} gave a short purely
combinatorial proof based on Shearer's work~\cite{Shearer85}. Since then, the 
lemma has yielded several improved bounds in combinatorics (see, e.g.,
\cite{BissacotFPS2011, BottcherKP2012, NdrecaPS12}).

We begin by stating the Cluster-Expansion Lemma. We then discuss its relationship
to the LLL (Theorem~\ref{theorem: LLL}), including its connection to the symmetric
form of the LLL. Finally, we revisit the Latin transversal problem from
Section~\ref{subsection: Latin transversals via the lopsided LLL} and show that the
Cluster-Expansion Lemma yields a stronger sufficient condition for the existence 
of a Latin transversal than that obtained there via the lopsided LLL.

\begin{theorem}[Cluster-Expansion Lemma {\cite{BissacotFPS2011}}]
\label{theorem: cluster-expansion lemma}
{\em Let $A_1,\ldots,A_n$ be a finite collection of events in a probability
space, and let $G=([n],E)$ be a dependency
graph for $\{A_i\}_{i=1}^n$. For each $i\in[n]$, define
\[
\Gamma_G(i)=\{j \in [n] \setminus \{i\}: \{i,j\} \in E\},
\qquad
\Gamma_G^+(i)=\Gamma_G(i)\cup\{i\}.
\]
For $S \subseteq [n]$, let $G[S]$ denote the subgraph of $G$
induced by the vertex set $S$, and let $\mathrm{Ind}(G[S])$
denote the family of independent sets of $G[S]$. Suppose that 
there exist numbers $y_1, \ldots, y_n > 0$ such that
\begin{align}
\label{eq:cluster-condition}
\mathbb{P}(A_i) \leq \frac{y_i}
{\displaystyle \sum_{I\in \mathrm{Ind}(G[\Gamma_G^+(i)])} \, \prod_{j\in I} y_j},
\qquad \forall\, i\in[n].
\end{align}
Then,
\begin{align}
\label{eq:cluster-lower-bound}
\mathbb{P}\left(\bigcap_{i=1}^n \overline{A_i}\right)
\geq \left( \sum_{I\in \mathrm{Ind}(G)} \prod_{i\in I} y_i
\right)^{-1}.
\end{align}}
\end{theorem}

\subsection{Connection between the Cluster-Expansion Lemma and the LLL}
\label{subsection: Connection between CEL and LLL}

The Cluster-Expansion Lemma can be viewed as a refinement of the LLL when the latter is 
formulated in terms of undirected dependency graphs. The following remarks show that the 
undirected-graph formulation of the LLL follows from a suitable relaxation of the 
Cluster-Expansion Lemma.

\begin{remark}[From digraph to graph]
\label{remark:digraph-to-graph}
{\em The LLL is formulated in Theorem~\ref{theorem: LLL} in terms of a dependency 
digraph $D=([n],E)$. Given such a digraph, one may associate an undirected 
graph $G=([n],E')$, defined by
\begin{align}
\label{eq: from digraph to graph}
\{i,j\}\in E' \quad \iff \quad (i,j)\in E \ \text{or}\ (j,i)\in E.
\end{align}
All notions of independence in Theorem~\ref{theorem: cluster-expansion lemma} are with
respect to this graph $G$.
It follows from \eqref{eq: from digraph to graph} that the neighborhood of a
vertex in $G$ consists of all vertices that are connected to it by either an
incoming or an outgoing edge in $D$. Consequently, the neighborhood of a vertex
in $G$ may be strictly larger than the set of vertices reachable from it by
outgoing edges in $D$. This distinction is relevant when comparing results
stated for dependency digraphs, such as the LLL, with results formulated for
undirected dependency graphs, such as the Cluster-Expansion Lemma.}
\end{remark}

\begin{remark}[Product-type lower bound and relation to the LLL]
\label{remark:product-LLL}
{\em Under the assumptions of Theorem~\ref{theorem: cluster-expansion lemma},
\begin{align}
\sum_{I \in \mathrm{Ind}(G)} \prod_{i \in I} y_i
& \leq \sum_{I \subseteq [n]} \prod_{i \in I} y_i \nonumber \\
\label{eq1: 24.04.26}
&= \prod_{i=1}^n (1+y_i),
\end{align}
where the last inequality holds since the summation on the left is restricted to the
independent sets of $G$, whereas the summation on the right ranges over all subsets
of $[n]$. Combining this with \eqref{eq:cluster-lower-bound} yields
\begin{align}
\mathbb{P}\left(\bigcap_{i=1}^n \overline{A_i}\right)
\geq \prod_{i=1}^n \frac{1}{1+y_i}.
\label{eq:product-bound}
\end{align}
Similarly, for all $i \in [n]$, since every independent set in the induced
subgraph $G[\Gamma_G^+(i)]$ is, in particular, a subset of $\Gamma_G^+(i)$,
we obtain
\begin{align}
\sum_{I \in \mathrm{Ind}(G[\Gamma_G^+(i)])} \, \prod_{j \in I} y_j
& \leq \sum_{I \subseteq \Gamma_G^+(i)}\prod_{j\in I} y_j \nonumber \\
& = \prod_{j \in \Gamma_G^+(i)}(1+y_j) \nonumber \\
\label{eq2: 24.04.26}
& = (1+y_i) \prod_{j: \, \{i,j\} \in E} (1+y_j),
\end{align}
and therefore (recall that $y_i > 0$)
\begin{align}
\label{eq1: 22.04.26}
\frac{y_i}{\displaystyle \sum_{I\in \mathrm{Ind}(G[\Gamma_G^+(i)])} \, \prod_{j\in I} y_j}
\geq \frac{y_i}{1+y_i} \; \prod_{j: \, \{i,j\} \in E} \frac1{1+y_j}.
\end{align}
By \eqref{eq1: 22.04.26}, the condition in \eqref{eq:cluster-condition} can be
replaced by the following stronger sufficient condition:
\begin{align}
\label{eq:cluster-stronger-condition}
\mathbb{P}(A_i) \leq \frac{y_i}{1+y_i} \; \prod_{j: \, \{i,j\} \in E} \frac1{1+y_j},
\qquad \forall\, i\in[n].
\end{align}
Setting $y_i := \frac{x_i}{1-x_i}$, with $x_i \in (0,1)$, gives
$\frac{1}{1+y_i} = 1 - x_i$ and $\frac{y_i}{1+y_i} = x_i$. Hence,
\eqref{eq:cluster-stronger-condition} and \eqref{eq:product-bound} 
reduce to \eqref{eq1: condition} and \eqref{eq1: LLL}, respectively,
recovering the standard form of the LLL.
The derivation above relies on a crude product bound that ignores
independence constraints, and is therefore not tight. In this sense, the lower
bound provided by the Cluster-Expansion Lemma reduces to that of the LLL after
the above relaxation. The Cluster-Expansion Lemma is formulated in terms of an
undirected dependency graph, whereas the LLL (Theorem~\ref{theorem: LLL}) is 
formulated in terms a dependency digraph. By Remark~\ref{remark:digraph-to-graph}, 
passing from a dependency digraph to its associated undirected graph may strengthen 
the sufficient conditions. Therefore, the above argument does not yield a derivation
of the LLL from the Cluster-Expansion Lemma, but rather highlights the relationship 
between the two results.}
\end{remark}

\subsection{Latin transversals via the Cluster-Expansion Lemma}
\label{subsection: Latin transversals - cluster expansion}

In Section~\ref{subsection: Latin transversals via the lopsided LLL}, we proved
Theorem~\ref{theorem: Latin transversals} by applying the lopsided LLL
to a family of bad events associated with a uniformly random permutation
of $[n]$. The same probabilistic model admits a stronger analysis via the
Cluster-Expansion Lemma, which yields an improved sufficient
condition for the existence of a Latin transversal.

\begin{theorem}[Proposition~4.2 of \cite{BissacotFPS2011}]
\label{theorem: BissacotFPS2011 - Proposition 4.2}
{\em Let $\mathbf{A}=(a_{i,j})$ be an $n\times n$ matrix in which, for every
symbol, the maximum number of its occurrences is at most $k$. If
\begin{align}
\label{eq: BissacotFPS2011 - Proposition 4.2}
k \leq \frac{27}{256} \; (n-1)+1,
\end{align}
then $\mathbf{A}$ contains a Latin transversal.}
\end{theorem}

\begin{proof}
We retain the notation and probabilistic construction from the proof of
Theorem~\ref{theorem: Latin transversals}.
Thus, $\pi$ is a uniformly random permutation of $[n]$,
$\mathcal{T}$ is the set of all ordered quadruples
$(i,j,i',j')$ such that $i<i'$, $j\neq j'$, and
$a_{i,j}=a_{i',j'}$, and
\[
A_{i,j,i',j'} = \bigl\{(\pi(i),\pi(i'))=(j,j')\bigr\},
\qquad (i,j,i',j') \in \mathcal{T}.
\]
As shown earlier,
\begin{align}
\label{eq: p-value, 18.06.26}
\mathbb P(A_{i,j,i',j'}) = \frac{1}{n(n-1)} =:p,
\qquad (i,j,i',j') \in \mathcal{T}.
\end{align}
Furthermore, the existence of a Latin transversal is equivalent to the validity of inequality
\eqref{eq1: 04.06.26}, so it suffices to prove the latter.

Let $G$ be the dependency graph on vertex set $\mathcal{T}$ introduced in the proof of
Theorem~\ref{theorem: Latin transversals}, whose adjacency condition is given by 
\eqref{eq: adjacency condition Latin transversal}). Fix a vertex $x=(i,j,i',j') \in \mathcal{T}$.
Following the argument in~\cite{BissacotFPS2011},
the neighborhood of $x$ can be covered by the four cliques
$C_i, \, C_{i'}, \, C_j, \, C_{j'}$,
where $C_i$ and $C_{i'}$ consist of all vertices in
$\mathcal{T}$ involving rows $i$ and $i'$, respectively, and
$C_j$ and $C_{j'}$ consist of all vertices in $\mathcal{T}$ 
involving columns $j$ and $j'$, respectively.
Indeed, if a vertex $(\ell,\ell',m,m')$ is adjacent to $x$, then by the adjacency condition
\eqref{eq: adjacency condition Latin transversal}, 
it must share one of the rows $i,i'$ or one of the columns $j,j'$.

As observed in the proof of Theorem~\ref{theorem: Latin transversals},
each of these four cliques contains at most $n(k-1)$ vertices.
Indeed, after fixing a row or a column, there are at most $n$ choices for
one entry and at most $k-1$ choices for a second entry containing the same 
symbol in a distinct row and column.

Let $\mu>0$, and set $y_x = \mu$ for all $x \in \mathcal{T}$. 
Since every independent set of the induced subgraph $G[\Gamma_G^{+}(x)]$
contains at most one vertex from each of the four cliques, we have 
\begin{align}
\sum_{I \in \mathrm{Ind}(G[\Gamma_G^+(x)])} \prod_{z \in I} y_z 
&= \sum_{I \in \mathrm{Ind}(G[\Gamma_G^+(x)])}  \mu^{|I|} \nonumber \\
\label{eq2: 18.06.26}
&\leq \bigl( 1 + \mu n (k-1) \bigr)^4.
\end{align}
Therefore, condition~\eqref{eq:cluster-condition} is satisfied whenever
\begin{align}
\label{eq3: 18.06.26}
p \leq \frac{\mu}{\bigl(1+n(k-1)\mu\bigr)^4}.
\end{align}
Substituting $p=\frac{1}{n(n-1)}$ from \eqref{eq: p-value, 18.06.26} and setting 
\begin{align}
\label{eq4: 18.06.26}
\mu :=\frac{\alpha}{n(k-1)}, \quad \alpha>0,
\end{align}
inequality \eqref{eq3: 18.06.26} becomes 
\begin{align}
\label{eq5: 18.06.26}
k-1 \leq \frac{\alpha \, (n-1)}{(1+\alpha)^4}, \quad \alpha>0.
\end{align}
Define $f(\alpha) = \frac{\alpha}{(1+\alpha)^4}$ for $\alpha>0$. Elementary calculus shows that
$f$ attains its maximum at $\alpha=\frac13$, where $f\bigl(\frac13\bigr) = \frac{27}{256}$.
Maximizing the right-hand side of \eqref{eq5: 18.06.26} over $\alpha > 0$, we recover  
inequality~\eqref{eq: BissacotFPS2011 - Proposition 4.2}. Consequently, by \eqref{eq: p-value, 18.06.26} 
and the choice $y_x = \mu$ for all $x \in \mathcal{T}$, condition~\eqref{eq:cluster-condition} 
is satisfied for the family of events $\{A_{i,j,i',j'}\}_{(i,j,i',j') \in \mathcal{T}}$.
By the Cluster-Expansion Lemma, inequality~\eqref{eq1: 04.06.26} holds. Hence, by the equivalence
established in \eqref{eq: equivalence - Latin transversal}, the matrix $\mathbf{A}$ contains a 
Latin transversal.
\end{proof}

\begin{remark}
\label{remark: Latin transversals cluster expansion}
{\em Theorem~\ref{theorem: BissacotFPS2011 - Proposition 4.2}
improves Theorem~\ref{theorem: Latin transversals} by replacing the
condition \eqref{eq: Latin transversals} with the weaker condition
\eqref{eq: BissacotFPS2011 - Proposition 4.2}. Since
$\tfrac{27}{256}\approx 0.10547$ and $\frac{1}{4e}\approx 0.09197$,
the cluster-expansion refinement permits a strictly larger admissible
range of values for $k$.
This improvement was one of the motivating examples in
\cite{BissacotFPS2011}, illustrating how the Cluster-Expansion Lemma
can exploit the clique structure of dependency neighborhoods more
effectively than the classical LLL when the dependency graph is undirected.
The reader is referred to the work of Harris and Srinivasan~\cite{HarrisS14}, 
which gave the first randomized polynomial-time algorithm for constructively 
finding a Latin transversal under the sufficient condition of 
Theorem~\ref{theorem: BissacotFPS2011 - Proposition 4.2}.}
\end{remark}

\section{The iterated Lov\'{a}sz Local Lemma}
\label{section: iterated LLL}

In the preceding sections, the LLL is applied as a one-shot probabilistic tool: one defines a collection
of bad events and proves that there is a positive probability that none of them occurs simultaneously. In many
combinatorial problems, however, a direct application of the lemma is either impossible or yields only weak
results because the dependencies among the bad events are too strong. A powerful extension of this idea is
provided by the \emph{iterated LLL}, also known as the \emph{semi-random method}, which is one of the cornerstones
of probabilistic combinatorics.

The origins of this methodology can be traced to R\"{o}dl's introduction of the nibble method in his work on
hypergraph packings and coverings \cite{Rodl1985}. Building on earlier probabilistic constructions and ideas,
including those of Ajtai, Koml\'{o}s, Pintz, Spencer, and Szemer\'{e}di \cite{AjtaiKPSS1982}, the nibble
method and related semi-random techniques evolved into a powerful framework that underlies many of the strongest
results in modern probabilistic combinatorics, particularly in graph and hypergraph coloring
\cite{Johansson1996,MolloyR2002,Molloy2019}.

The basic philosophy is to construct the desired object gradually. Instead of making all random choices at once,
one performs a partial random construction, removes the resulting conflicts, and then repeats the procedure on the
remaining unresolved portion of the problem. The residual structure typically becomes simpler after each stage,
eventually reducing the problem to a setting where a final application of the LLL or a deterministic argument
completes the construction.

A simple example arises in graph coloring. Let $G$ be a graph with maximum degree $\Delta$, and suppose that
one wishes to produce a proper coloring using relatively few colors. Rather than coloring all vertices at once,
consider the following iterative procedure. During a given round, every uncolored vertex independently selects
a tentative color from a fixed set of colors. If two adjacent vertices are assigned the same color, their tentative
assignments are discarded. Vertices that are not involved in any conflict keep their colors permanently, and the
procedure is repeated on the subgraph induced by the remaining uncolored vertices.

The key question is whether the residual graph becomes significantly simpler after one round. To analyze this, one
studies local parameters such as the number of uncolored neighbors of a given vertex $v$ after the conflict-resolution
step. By linearity of expectation, one first shows that the expected residual degree of $v$ is smaller than its original
degree, often by a fixed multiplicative factor. However, expectation alone is insufficient; one must also show that most
vertices behave close to this average. Appropriate concentration inequalities are therefore used to show that large
deviations from the expected behavior are unlikely, occurring with exponentially small probability; see
\cite{BLM2013, RaginskySason2013} for treatments of the subject.

One then defines a bad event $A_v$ to be the event that the residual degree of $v$ exceeds a prescribed threshold. Although
each bad event has small probability, the events are not independent because nearby vertices are influenced by many of the
same random color choices. Nevertheless, each event depends only on random choices within a bounded neighborhood of $v$, and
hence each event is dependent on only a limited number of others. The LLL can therefore be applied to show that, with positive
probability, none of the events $\{A_v\}$ occurs. Consequently, there exists a round in which every vertex simultaneously
experiences the prescribed reduction in residual degree.

This combination of concentration inequalities and the LLL is the hallmark of the semi-random method. Concentration inequalities
provide local control by showing that undesirable deviations are individually unlikely, while the LLL converts these local estimates
into a global statement asserting that all vertices satisfy the required property simultaneously. As a result, after one iteration, 
the maximum degree of the residual graph is significantly smaller than that of the original graph.

The procedure can then be repeated. If the maximum degree decreases by a constant factor in each round, then after $O(\log \Delta)$
iterations the residual graph has bounded degree. At that stage, the remaining coloring problem can often be completed by elementary
methods or by a final application of the LLL. This iterative reduction of complexity is the essence of the semi-random method.

The semi-random method should also be distinguished from the algorithmic framework of Moser and Tardos in Section~\ref{section: Moser--Tardos Algorithm}.
Recall that, in the latter approach, one begins with a complete random assignment and repeatedly resamples the variables involved in occurring bad events.
By contrast, the semi-random method constructs the desired object incrementally, repeatedly simplifying the remaining instance through a sequence
of random and deterministic steps. Although both approaches originate from the ideas underlying the LLL, they employ fundamentally different mechanisms
and have led to distinct developments in probabilistic combinatorics.

\section{Outlook}
\label{section: outlook}

Fifty years after its inception in \cite{LLL75}, the LLL
continues to pose fundamental conceptual and technical
challenges. While Shearer’s criterion \cite{Shearer85} (Theorem~\ref{theorem: Shearer})
provides a sharp characterization of the existential regime, an equally complete
understanding of the algorithmic landscape remains elusive.

The algorithmic aspects of the LLL are now well understood in the
variable framework, most notably through the Moser--Tardos resampling algorithm
\cite{MoserTardos10} and its subsequent analyses (see \cite{HeLLWX26} and references
therein). Despite substantial progress, the situation is less clear in more general
settings, particularly in the presence of arbitrary dependency structures. The
diversity of modern formulations, ranging from lopsided and resampling-oracle variants
to commutativity-based conditions \cite{Szegedy2013}, suggests that a unifying
perspective is still lacking.

Moreover, important directions remain only partially understood. One direction
concerns the complexity of distributed and local algorithms for constructing
configurations guaranteed by the lemma, particularly near the threshold of its
applicability. Another direction concerns approximate counting and sampling in
parameter regimes where the LLL guarantees the existence of configurations
avoiding all bad events. A further direction seeks to extend and refine the
lemma in more general settings, including measurable probability spaces and
non-commutative (quantum) frameworks. While substantial progress has been made
on measurable versions of the LLL and related measurable graph-coloring problems,
notably through the work of Bernshteyn \cite{Bernshteyn2019}, a complete understanding
of the scope, limitations, and optimal conditions of such extensions has yet to emerge.
In each of these directions, many fundamental questions concerning sharp thresholds,
constructive methods, and the precise boundaries of applicability remain open.

\section*{Appendices}

\appendix

\section{Completion of the proof of Corollary~\ref{corollary: 1/2}}
\label{appendix: 1/2}

\begin{lemma}
\label{lemma: 1/2}
{\em For every $d \in \mathbb{N}$,
\begin{align}
\label{eq1: 04.04.26}
\frac{d^d}{(d+1)^{d+1}} > \frac{1}{e\bigl(d+\frac12\bigr)}.
\end{align}
Moreover, $\alpha = \frac12$ is the smallest constant such that
\begin{align}
\label{eq2: 04.04.26}
\frac{d^d}{(d+1)^{d+1}} \geq \frac{1}{e(d+\alpha)}, \quad \forall \, d \in \mathbb{N}
\end{align}
holds.}
\end{lemma}
\begin{proof}
Define $\phi(x) := x \ln x - (x+1) \ln(x+1) + 1 + \ln\bigl(x+\tfrac12\bigr)$ for $x>0$.
Then for each $x>0$, the inequality $\phi(x)>0$ is equivalent to
$\frac{x^x}{(x+1)^{x+1}} > \frac{1}{e \bigl(x+\frac12 \bigr)}$.
Therefore, to prove \eqref{eq1: 04.04.26}, it suffices to show that $\phi(x) > 0$ for all $x > 0$. 
Differentiating and setting $t:= \frac1x > 0$, we obtain $\phi'(x) = -\ln(1+t) + \frac{2t}{2+t}$. 
To show that $\phi'(x)<0$, consider the auxiliary function $f(t) := \ln(1+t) - \frac{2t}{2+t}$ for all $t \geq 0$. 
Since $f(0) = 0$ and $f'(t) = \frac{t^2}{(1+t)(2+t)^2} > 0$ for all $t>0$, the function $f$ is strictly increasing 
on $[0, \infty)$. Hence $f(t)>0$ for all $t>0$, which proves that $\phi'(x) < 0$ for all $x>0$.
Therefore, $\phi$ is strictly decreasing on $(0,\infty)$, and
\begin{align*}
\lim_{x \to \infty} \phi(x) = \lim_{x \to \infty} \Biggl\{ \ln \Biggl(1 + \frac1{2x} \Biggr) - (x+1) \ln \Biggl(1+\frac1x \Biggr) + 1 \Biggr\} = 0,
\end{align*}
where the second limit holds by the standard expansion $\ln(1+u) = u + O(u^2)$.
Since $\phi$ is strictly decreasing and tends to $0$ at infinity, it follows that $\phi(x) > 0$ for all $x>0$. In particular,
inequality \eqref{eq1: 04.04.26} holds for all $d \in \mathbb{N}$.

\noindent
It remains to prove that the constant $\alpha = \frac12$ is the smallest value for which inequality \eqref{eq2: 04.04.26} holds.
Suppose that \eqref{eq2: 04.04.26} holds for some $\alpha>0$. Then,
\begin{align*}
\alpha \geq \frac{d+1}{e} \, \biggl(1+ \frac1d \biggr)^d - d, \quad \forall \, d \in \mathbb{N}.
\end{align*}
Using the asymptotic expansion as $d \to \infty$
\[
\biggl(1+\frac1d\biggr)^d = e \, \Biggl( 1-\frac{1}{2d} + O\biggl(\frac{1}{d^2}\biggr) \Biggr),
\]
it follows that $\alpha \geq \tfrac12$. Since \eqref{eq1: 04.04.26} shows that $\alpha=\tfrac12$ is admissible,
it is the smallest admissible value.
\end{proof}

\section{\texorpdfstring{Derivation of \eqref{eq: k_0}}
                     {Derivation of (k0)}}
\label{appendix: W}
\noindent
Let
\vspace*{-1.5em}
\begin{align}
\label{eq: a, C}
a:=-\ln(1-\varepsilon)>0, \qquad C:= \sqrt{\frac{3}{8 \pi}} \, e^3.
\end{align}
Then inequality \eqref{eq3: 12.03.26} can be rewritten as
\[
-\frac{2a}{3} \, k \; e^{-\frac{2a}{3} \, k}
\geq -\frac{2a}{3}\left(\frac{e^{-2a}}{C}\right)^{2/3}.
\]
Applying the branch $W_{-1}$ of the Lambert $W$ function to both sides, we obtain
\begin{align}
\label{eq1: 06.06.26}
k \geq -\frac{3}{2a}\,
W_{-1}\!\left(
-\frac{2a}{3}\left(\frac{e^{-2a}}{C}\right)^{2/3}
\right).
\end{align}
The branch $W_{-1}$ is used rather than the principal branch $W_0$
since the argument of the Lambert $W$ function is negative and tends
to $0$ as $\varepsilon \to 0$. The branch $W_0$ would yield a bounded
solution for $k$, whereas the desired solution corresponds to the
large-$k$ regime; this is captured by the branch $W_{-1}$, for which
$W_{-1}(x)\to -\infty$ as $x\to 0^{-}$. Substituting the values of $a$ 
and $C$ from \eqref{eq: a, C} into \eqref{eq1: 06.06.26} gives 
\begin{align}
\label{eq2: 06.06.26}
k \geq \left(\frac{3}{2\ln(1-\varepsilon)}\right) \;
W_{-1}\left(
\frac{4}{3e^2} \sqrt[3]{\frac{\pi}{3}} \,
\ln(1-\varepsilon) \; (1-\varepsilon)^{4/3} \right).
\end{align}
Combining inequality~\eqref{eq2: 06.06.26} with the requirement $k \geq 3$ in \eqref{eq1: 12.03.26}
finally gives, together with \eqref{eq: beta}, the valid selection of the integer $k_0$
in \eqref{eq: k_0}.

\subsection*{Use of Generative-AI tools declaration}
While writing this expository paper, the author used ChatGPT (OpenAI) solely for language editing 
and stylistic refinement of the manuscript. The author assumes full responsibility for its content.

\subsection*{Acknowledgments}
The author gratefully acknowledges the referees for their timely reports and constructive comments,
which helped improve the presentation of this paper, and Ugo Vaccaro for raising a question that led 
to Corollary~\ref{corollary: 1/2}.

\subsection*{Conflict of interest}
The author declares no conflicts of interest.

\end{document}